\documentclass[11pt,a4paper]{amsart}
\usepackage[utf8]{inputenc}
\usepackage[english]{babel}
\usepackage{amsmath}
\usepackage{amsfonts}
\usepackage{amsthm}
\usepackage{color}
\usepackage{amssymb}
\usepackage{mathrsfs, ccfonts, mathpazo, mathrsfs}
\usepackage{enumitem,linegoal}
\usepackage{tikz}
\usepackage{dsfont}
\usepackage{enumitem} 
\usepackage{calc}
\usepackage{stmaryrd}
\usepackage[mathscr]{eucal}
\usepackage{xcolor}
\usepackage{hyperref}
\definecolor{mediumtealblue}{rgb}{0.0, 0.33, 0.71}
\definecolor{tangelo}{rgb}{0.98, 0.3, 0.0}
\hypersetup{%
	colorlinks=true, 
	linkcolor= mediumtealblue,
	anchorcolor=black, 
	citecolor= red, 
	urlcolor = blue
}
\usepackage[nameinlink,capitalize,noabbrev]{cleveref}
\usepackage[left=2cm,right=2cm,top=2cm,bottom=2cm]{geometry}
\newtheorem{theorem}{Theorem}[section]
\newtheorem{defn}[theorem]{Definition}
\newtheorem{lem}[theorem]{Lemma}
\newtheorem{rem}[theorem]{Remark}

\newtheorem{prop}[theorem]{Proposition}
\newtheorem{cor}[theorem]{Corollary}

\newcommand{\erjp}{\mathbf{e}^{j+1}}
\newcommand{\erj}{\mathbf{e}^{j}}
\newcommand{\ru}{\mathrm{U}}
\newcommand{\mujp}{\mathcal{U}^{j+1}}
\newcommand{\muj}{\mathcal{U}^j}
\newcommand{\ujp}{\ru^{j+1}}
\newcommand{\uj}{\ru^j}

\newcommand{\EE}{\mathbb{E}}
\newcommand{\mo}{\mathscr{O}}

\newcommand{\h}{\mathrm{H}}
\newcommand{\ve}{\mathrm{V}}

\newcommand{\el}{\mathrm{L}}

\newcommand{\err}{\mathbb{R}}
\newcommand{\eps}{\varepsilon}
\newcommand{\bu}{\mathbf{u}}
\newcommand{\bv}{\mathbf{v}}

\newcommand{\by}{\mathbf{y}}
\newcommand{\bb}{\mathbf{B}}

\newcommand{\mb}{\mathrm{B}}

\newcommand{\bx}{\mathbf{x}}

\newcommand{\FF}{\mathbb{F}}
\newcommand{\PP}{\mathbb{P}}

\newcommand{\rA}{\mathrm{A}}
\newcommand{\bw}{\mathbf{w}}
\DeclareMathOperator{\trace}{\mathrm{Tr}}
\newcommand{\bW}{\boldsymbol{W}}
\newcommand{\bee}{\boldsymbol{e}}

\newcommand{\cF}{\mathscr{F}}
\newcommand{\cL}{\mathscr{L}}

\author{Hakima Bessaih}
\author{Erika Hausenblas}
\author{Tsiry Randrianasolo}
\author{Paul Razafimandimby}

\address[Hakima Bessaih]{University of Wyoming, Department of mathematics, Laramie, WY 82071-3036, United-States}

\address[Erika Hausenblas]{Montanuniversit\"at Leoben, Lehrstuhl f\"ur angewandte mathematik, Peter-Tunner-Stra\ss e 25-27, 8700 Leoben, Austria}

\address[Tsiry Randrianasolo]{Bielefeld University, Department of applied mathematics, Universitätsstra\ss e 25, 33615 Bielefeld, Germany}
\email[Corresponding author]{\href{mailto:trandria@math.uni-bielefeld.de}{trandria@math.uni-bielefeld.de}}

\address[Paul Razafimandimby]{Department of Mathematics and Applied Mathematics, University of Pretoria, Lynwood Road, Hatfield, Pretoria 0083, South Africa}

\title[Numerical approximation of stochastic evolution equations]{Numerical approximation of stochastic evolution equations: Convergence in scale of Hilbert spaces}
\date{\today}
\thanks{
	The research of Hakima Bessaih was supported by the NSF grant DMS-1418838.
	The research of Erika Hausenblas  was supported by the Austrian Science Fund (FWF): P25968.
	The research of Tsiry Randrianasolo was supported by the Austrian Science Fund (FWF): P25968 and the German Research Council as part of the Collaborative Research Center SFB 1283.}
\begin{document}
	\begin{abstract}
		The present paper is devoted to the numerical approximation of an abstract stochastic nonlinear evolution equation in a separable Hilbert space {$\h$}.  Examples of  equations  which fall into our framework include the GOY and Sabra shell models and { a class of nonlinear heat equations.} The space-time numerical scheme is defined in terms of a Galerkin approximation in space and  a { semi-implicit Euler--Maruyama scheme in time}. {We prove the convergence in probability of our scheme by means of an estimate of the error on a localized set of arbitrary large probability.} Our error estimate is shown to hold in a more regular space $\ve_{\beta}\subset \h$ with $\beta \in [0,\frac14)$ and { that the explicit rate of convergence of our scheme depends  on this parameter $\beta$. }
	\end{abstract}
	\maketitle
	\section{Introduction}
	\label{sec:introduction}
	Throughout this paper we fix a complete filtered probability space $\mathfrak{U}=(\Omega, \cF, \FF, \PP)$ with the filtration $\FF= \{\cF_t;t\geq 0\}$ satisfying the usual  conditions. {We also  fix } a separable Hilbert space $\h$ equipped with a scalar product $(\cdot,\cdot)$ with the associated norm {$| \cdot |$} and  another separable Hilbert space $\mathscr{H}$.
	In this paper, we analyze numerical approximations { for an abstract stochastic evolution equation  of the form}

	\begin{equation}
	\label{spdes}
	\left\{
	\begin{split}
	d\bu &= -[\rA\bu + \mb(\bu,\bu)]dt + G(\bu)dW,~t\in[0,T],
	\\
	\bu(0)&=\bu_0,
	\end{split}
	\right.
	\end{equation}

	where hereafter $T>0$ is {a fixed number} and  $\rA$ is a self-adjoint positive operators on $\h$. { The operators $\mb$} and $G$ are nonlinear maps satisfying several technical assumptions to be specified later and $W= \{W(t); 0\leq t\leq T\}$ is a $\mathscr{H}$-valued Wiener process. 

	The abstract equation~\eqref{spdes} can describe several problems from different fields including mathematical finance, electromagnetism, and fluid dynamic. {Stochastic models have been widely used to describe small fluctuations or perturbations which arise in nature. 
		For a more exhaustive introduction to the importance of stochastic models and the analysis of stochastic partial differential equations, we refer the reader  to \cite{chow,hairer2009introduction,krylov+rozoskii,prevot+roeckner,revuz+yor}.}
	
	Numerical analysis for stochastic partial differential equations (SPDEs) has known a strong interest in the past decades. Many algorithms which are based on either finite difference or finite element methods or spectral  Galerkin methods  (for the space discretization) and on either Euler schemes or Crank-Nicholson or Runge-Kutta schemes (for the temporal discretization) have been introduced for both the linear and nonlinear cases and their rate of convergence have been investigated widely. {Here we should note that  the  orders of convergence that are frequently analyzed are the weak and strong orders of convergence.}  The literature on numerical analysis for SPDEs is now very extensive.  Without being exhaustive,  we only cite amongst other the recent papers \cite{Lindneretal-2013, Carellietal-2012, AD+JP-2009, AA+SL-2016,EC+EH+AP}, the excellent review paper \cite{Jentzen-Review} and references therein. Most of the literature deals with the stochastic heat equations with globally Lipschitz nonlinearities, but there are also several papers that treat abstract stochastic evolution equations. For example, Gyongy and Millet in \cite{gyongy2009rate} { investigated a general evolution equation with an operator that has the strong monotone and global} Lipschitz properties. They were able to implement a space-time discretization and showed a rate of convergence in mean under appropriate assumptions.  Similar rate of convergence have been {obtained} by Bessaih and Schurz in \cite{bessaih2007upper} for  an equation with globally Lipschitz nonlinearities. { When a system of SPDEs with non-globally Lipschitz nonlinearities, such as the stochastic Navier-Stokes equations,  is considered the story is completely different}. Indeed, in this case the rate of convergence obtained is generally only in probability. {This kind of  convergence was introduced for the first time by Printems in \cite{Printems} and is  well suited for SPDEs with locally Lipschitz coefficients. When the stochastic perturbation is in an additive form (additive noise), then using a path wise argument one can prove a convergence in mean, we refer to  Breckner in \cite{Breckner}. Let us mention that in this case, no rate of convergence can be deduced. }
	
	{Recent literature involving nonlinear models with nonlinearities which are locally Lipschitz  are   \cite{Dorsek,Carelli+Prohl, ZB+EC+AP-IMA, bessaih2014splitting} and references therein}.  In \cite{ZB+EC+AP-IMA} { martingale solutions} to  the incompressible
	Navier-Stokes equations with Gaussian multiplicative noise are constructed from a finite element based space-time { discretizations}. The authors of \cite{Carelli+Prohl} proved the convergence in probability {with} rates of an explicit and an implicit numerical schemes by means of a {Gronwall argument}. The main issue when the term $\mb$ is not {globally} Lipschitz lies on
	its interplay with the stochastic forcing, which prevents a Gronwall argument in the context of expectations. This issue is for example solved in \cite{Breckner,Caraballo} by the introduction of a weight, which when carefully chosen contributes in removing unwanted terms and allows to use Gronwall lemma. In \cite{Carelli+Prohl}, the authors use different approach by computing the error estimates on a sample subset $\Omega_k\subset\Omega$ with large probability. In particular, the set $\Omega_k$ is carefully chosen so that the random variables $\lVert \nabla\bu^\ell\rVert_{L^2}$ are bounded as long as the events are taken in $\Omega
	_k$, and $\lim_{k\searrow 0} \PP(\Omega\setminus\Omega_k)= 0$.  The result is then obtained using standard arguments based on the Gronwall lemma. 
	Other kinds of numerical algorithms have been used  in \cite{bessaih2014splitting} for a 2D stochastic Navier-Stokes equations. There, a {splitting up method} has been used and a rate of convergence in probability is obtained. A blending of a splitting scheme and the method of cubature on Wiener space applied to a spectral Galerkin discretisation of degree $N$ is used in \cite{Dorsek} to approximate the marginal distribution of the solution of the stochastic Navier- Stokes equations on the two-dimensional torus and rates of convergence are also given. For the numerical analysis of other kind of stochastic nonlinear models that enjoy the local  Lipschitz condition,  
	without being exhaustive, we refer to \cite{AD+JP-2006,AD+JP-1999,DB+AJ,JC+JH+ZL} and references therein. They include the {stochastic} Schr\"odinder, Burgers and KDV  equations.  
	
	In the present  paper, we are interested in the numerical treatment of the abstract stochastic evolution equations \eqref{spdes}. We first give a simple and short proof of the existence and uniqueness of a mild solution and study the regularity of this solution.  The result about the existence of solution is based on a fixed point argument recently developed in \cite{BHR-14}.  Then, we discretize \eqref{spdes} using a coupled Galerkin method and (semi-)implicit Euler scheme and show { convergence in probability with rates} {in $\ve_{\beta}:= D(\rA^\beta)$.} {Regarding our approach it is similar to \cite{Carelli+Prohl} and \cite{Printems}, however, the results are different. Indeed, while \cite{Carelli+Prohl} and \cite{Printems} establish their rates of convergence in the space $\h$ where the solution lives, we establish our rate of convergence in $\ve_\beta\subset \h$ where $\beta \in [0,\frac14)$ is arbitrary.} Hence, our result does not follow from the papers \cite{Carelli+Prohl} and \cite{Printems}. In contrast to the nonlinear term of Navier--Stokes equations with periodic boundary condition treated in \cite{Carelli+Prohl}, our nonlinear term  does not satisfy the property $\langle \mb(\bu,\bu),\rA\bu \rangle=0$ which plays a crucial role in the analysis in \cite{Carelli+Prohl}.  We should also point out  that our model  does not fall into the general framework of the papers \cite{gyongy2009rate} and \cite{bessaih2007upper},  see Remark \ref{REM-NEW}.

	Examples of semilinear equations  which fall into our framework include the GOY and Sabra shell models.
	{These toy models are used to mimic some features of {turbulent flows.} It seems that our work is the first one rigorously addressing the {numerical approximation of such models}.  Our result also confirm that, in term of numerical analysis, shell models behave far better than the Navier-Stokes equations. On the theoretical point of view, we provide a new and simple proof of the existence of solutions to stochastic shell models driven by Gaussian multiplicative noise. On the physical point of view, 
		it is also worth mentioning that shell models of turbulence are toy models which consist of infinitely many nonlinear differential equations
		having a structure similar to the Fourier representation of the
		Navier-Stokes equations, see \cite{Ditl}. Moreover, they capture quite well the statistical properties of three dimensional Navier-Stokes equations, like the Kolmogorov energy spectrum and the intermittency scaling exponents for the high-order
		structure functions, see  \cite{Ditl} and \cite{EG-2012}. }	
	Due to their success  in the study of turbulence, new shell models have been derived  by several prominent physicists for the investigation of the turbulence in magnetohydrodynamics, see for instance \cite{plunian2013shell}.
	
	Another example {of system of  equations which falls into our framework is a class of nonlinear heat equations described in Section~\ref{motivations}}.  We do not know whether our results can cover the numerical analysis of 1D  stochastic nonlinear heat equations driven by additive space-time noise.  Despite this fact we believe that our paper is still interesting as we are able to treat a class of 2D stochastic nonlinear heat equations with locally Lipschitz coefficients and we are not aware of results similar to ours. In fact, most of results related to stochastic heat equations are either about 1D model, or $d$-dimensional, $d\in \{1,2,3\}$,  models with globally Lipschitz coefficients and deal with weak convergence or convergence in weaker norm, see for instance \cite{Lindneretal-2013, Carellietal-2012, AD+JP-2009, AA+SL-2016}.

	This paper is organized as follows: in Section~\ref{method-and-materials},  we introduce the necessary notations and the standing assumptions that will be used in the present work. In Section~\ref{Sec-Stabl}, we present our numerical scheme and also discuss the stability and existence of solution at each time step. The convergence of the proposed method  is presented in Section~\ref{Sec-Proof}. In Section~\ref{motivations} we present  the stochastic shell models for turbulence and a class of stochastic nonlinear heat equations as motivating examples.
	
	\section{Notations, assumptions, preliminary results and the main theorem}\label{method-and-materials}
	In this section we introduce the necessary notations and the standing assumptions that will be used in the present work. We will also introduce our numerical scheme and state our main result.
	\subsection{Assumptions and notations}
	Throughout this paper we fix a separable Hilbert space $\h$ with norm $\vert \cdot \vert$ and a fixed orthonormal basis $\{\psi_n; n\in \mathbb{N}\}$. 
	We assume that we are given a linear operator $\rA:D(\rA)\subset \h \to \h$ which is a self-adjoint and positive operator such that the fixed  orthonormal basis $\{\psi_n; n\in \mathbb{N}\}$ satisfies $$\{\psi_n; n\in \mathbb{N}\}\subset D(\rA), \quad \rA \psi_n= \lambda_n \psi_n, $$
	for an increasing sequence of positive numbers $\{\lambda_n; n\in \mathbb{N}\}$  with $\lambda_n \to \infty$ as $n\nearrow \infty$. It is clear that $-\rA$ is the infinitesimal generator of  an analytic semigroup $e^{-tA}, t\ge0$, on $\h$.
	For any $\alpha \in \err$ the domain of $\rA^\alpha$ denoted by $\ve_\alpha=D(\rA^\alpha)$ is a separable Hilbert space when equipped with the scalar product
	\begin{equation}
	(( \bu, \bv ))_\alpha=\sum_{k=1}^{\infty} \lambda_k^{2\alpha} \bu_k {\bv}_k, \text{ for } \bu,\; \bv\in \ve_\alpha. 
	\end{equation}
	The norm associated to this scalar product will be denoted by $\Vert \bu \Vert_\alpha$, $\bu \in \ve_\alpha$.
	In what follows we set $\ve:=D(\rA^\frac12)$. 
	
	Next, we consider a nonlinear map $\mb(\cdot, \cdot): \ve \times \ve \to \ve^\ast$ satisfying the following set of assumptions, where hereafter {$\ve^\ast$ denotes the dual of the Banach space $\ve$.}
	\begin{enumerate}[label=(\textbf{B\arabic{*}})]
		\item[]
		\item \label{NB-a}	 There exists a constant $C_0>0$ such that for any $\theta\in [0,\frac12)$ and $\gamma\in (0,\frac12)$ satisfying $\theta+\gamma \in (0,\frac12]$, we have
		\begin{equation}\label{Eq-nonlin-B}
		\Vert \mb(\bu,\bv)-\mb(\bx,\by)\Vert_{-\theta}\le  
		\begin{cases}
		C_0 \Vert \bu -\bx \Vert_{\frac12-(\theta+\gamma)} (\Vert \bv \Vert_\gamma + \Vert \by\Vert_\gamma )+ \Vert \bv-\by\Vert_\gamma (	\Vert \bu\Vert_{\frac12-(\theta+\gamma)} +\Vert \bx \Vert_{\frac12-(\theta+\gamma)})\\ 
		\;\;\text{for any } \bu,\bx \in \ve_{\frac12-(\theta+\gamma)}\text{ and } \bv, \by \in \ve_{\gamma},\\
		C_0 (\Vert \bu \Vert_\gamma +\Vert \bx\Vert_\gamma ) \Vert \bv-\by \Vert_{\frac12-(\theta+\gamma)}+  \Vert \bu-\bx\Vert_\gamma (	\Vert \bv\Vert_{\frac12-(\theta+\gamma)} +\Vert \by \Vert_{\frac12-(\theta+\gamma)}) \\ 
		\;\;\text{for any } \bv, \by \in \ve_{\frac12-(\theta+\gamma)} \text{ and } \bu, \bx  \in \ve_{\gamma}.
		\end{cases}
		\end{equation} 
		Due to the continuous embedding $\ve_{-\theta}\subset \ve_{-\frac12}$, $\theta \in [0,\frac12)$, 
		\eqref{Eq-nonlin-B} holds with $\theta$ and $\frac12-(\theta+\gamma)$ respectively replaced by $\frac 12$ and $\frac12-\gamma$ where $\gamma >0$ is arbitrary.
		
		In addition to the above, we assume that for any $\eps>0$ there exists a constant $C>0$ such that 
		\begin{equation}\label{Eq-nonlin-B-H2}
		\lvert \mb(\bu,\bv)\rvert\le C \lvert \bu \rvert \lVert \bv \rVert_{\frac12+\eps}, \text{ for any } \bu \in \h, \bv \in \ve_{\frac12+\eps}.
		\end{equation}
		\item \label{NB-b} We also assume that for any $ \bu, \bv \in \ve$
		\begin{equation}
		\langle \rA \bv+ \bb(\bu,\bv),\bv\rangle \ge  \Vert \bv \Vert^2_\frac12.
		\end{equation}
		\item \label{NB-c} We assume that for any $\bu \in \h$ we have 
		\begin{equation}
		\mb(0,\bu)=\mb(\bu,0)=0.
		\end{equation}

	\end{enumerate}
	Note that Assumptions \ref{NB-a} and \ref{NB-c} imply
	\begin{enumerate}[label=(\textbf{B\arabic{*}})$^\prime$]
		\item \label{Assum-B1p}	There exists a constant $C_0>0$ such that for any numbers $\theta\in [0,\frac12)$ and $\gamma\in (0,\frac12)$ satisfying $\theta+\gamma \in {(0,\frac12]}$, we have 
		\begin{equation}\label{Eq-nonlin-B-2}
		\Vert \mb(\bu,\bv)\Vert_{-\theta}\le C_0 \begin{cases}
		\Vert \bu \Vert_{\frac12-(\theta+\gamma)} \Vert \bv \Vert_\gamma \;\; \text{ for any } \bu \in \ve_{\frac12-(\theta+\gamma)}\text{ and } \bv\in \ve_{\gamma},\\
		\Vert \bu \Vert_\gamma \Vert \bv\Vert_{\frac12-(\theta+\gamma)} \text{ for any } \bv \in \ve_{\frac12-(\theta+\gamma)}, \text{ and } \bu \in \ve_{\gamma}.
		\end{cases}
		\end{equation} 
		If $\theta=\frac12$, then \eqref{Eq-nonlin-B-2} holds with $\frac12-(\theta+\gamma)$ replaced by $\frac12-\gamma$ where $\gamma>0$ is arbitrary.
	\end{enumerate}
	Let $\{\mathfrak{w}_j; \; j \in \mathbb{N} \}$ be a sequence of mutually independent and identically distributed standard Brownian motions  on $\mathfrak{U}$. Let $\mathscr{H}$ be separable Hilbert space and $\mathscr{L}_1(\mathscr{H})$ be the space of all trace class operators on $\mathscr{H}$. 
	Recall that if  $Q\in \mathscr{L}_1(\mathscr{H})$ is a symmetric, positive operator and $\{\varphi_j; j\in \mathbb{N} \}$ is an orthonormal basis of $\mathscr{H}$ consisting of eigenvectors of $Q$, then the series
	$$ W(t)= \sum_{j=1}^\infty \sqrt{q_j} \mathfrak{w}_j(t) \varphi_j, \quad t\in [0,T], $$ where  $\{q_j;\; j\in \mathbb{N}\}$ are the eigenvalues of $Q$,   
	converges in $\el^2(\Omega; C([0,T]; \mathscr{H}))$ and it defines an $\mathscr{H}$-valued Wiener process with covariance operator $Q$. Furthermore,  for any positive integer $\ell>0$ there exists a constant $C_\ell >0$ such that
	\begin{equation}\label{Increment}
	\EE \Vert W(t)- W(s)\Vert^{2\ell}_{\mathscr{H}}\le C_\ell \vert t-s\vert^\ell \left(\trace Q\right)^\ell, 
	\end{equation} for any $t,s\ge 0$ with $t\neq 0$. 
	Before proceeding further we recall few facts about stochastic integral.
	Let $\mathrm{K}$ be a separable Hilbert space, $\mathscr{L}(\mathscr{H}, \mathrm{K})$ be the space of all bounded linear $\mathrm{K}$-valued  operators defined on $\mathscr{H}$, $\mathscr{M}_T^2(\mathrm{K})$ be the space of all equivalence classes of $\mathbb{F}$-progressively measurable processes $\Psi: \Omega\times [0,T]\to\mathrm{K}$ satisfying 
	$$ \EE\int_0^T \Vert \Psi(s)\Vert^2_{\mathrm{K}}ds <\infty.$$ 
	If $Q\in \mathscr{L}_1(\mathscr{H})$ is a symmetric, positive and trace class operator then $Q^\frac12\in \mathscr{L}_2(\mathscr{H})$ and for any $\Psi \in \mathscr{L}(\mathscr{H}, \mathrm{K})$ we have $\Psi \circ Q^\frac12 \in \mathscr{L}_2(\mathscr{H}, \mathrm{K}) $, where  $\mathscr{L}_2(\mathscr{H}, \mathrm{K})$ (with $\mathscr{L}_2(\mathscr{H}):=\mathscr{L}_2(\mathscr{H}, \mathscr{H})$) is the Hilbert space of all operators $\Psi\in \mathscr{L}(\mathscr{H}, \mathrm{K} )$ satisfying 
	$$ \Vert \Psi\Vert_{\mathscr{L}_2(\mathscr{H},\mathrm{K})}^2 =\sum_{j=1}^\infty \Vert \Psi\varphi_j\Vert^2_{\mathrm{K}}<\infty. $$   
	Furthermore, from the theory of stochastic integration on infinite dimensional Hilbert space, see \cite{DP+JZ-92},
	{
		for any
		$\mathscr{L}(\mathscr{H},\mathrm{K})$-valued process $\Psi$ such that  $\Psi\circ Q^{1/2}\in \mathscr{M}^2_T(\mathscr{L}_2(\mathscr{H},\mathrm{K}))$
	} the process $M$ defined by 
	$$ M(t) =\int_0^t \Psi(s)dW(s), t\in [0,T],$$ is a $\mathrm{K}$-valued martingale. Moreover,  we have the following It\^o's isometry 
	\begin{equation}
	\EE \biggl(\biggl\Vert \int_0^t \Psi(s) dW(s)\biggr\Vert^2_{\mathrm{K}}\biggr) =\EE \biggl(\int_0^t \Vert \Psi(s) Q^\frac12 \Vert^2_{\mathscr{L}_2(\mathscr{H}, \mathrm{K})} ds \biggr), \forall t\in [0,T],
	\end{equation}	
	and the Burkholder-Davis-Gundy inequality 
	\begin{equation}
	\EE\biggl(\sup_{0\le s\le t}\biggl \Vert \int_0^s \Psi(\tau) dW(\tau)\biggr\Vert_ {\mathrm{K}}^q\biggr) \le C_q \EE \biggl( \int_0^t \Vert \Psi(s) Q^\frac12 \Vert^2_{\mathscr{L}_2(\mathscr{H},\mathrm{K})}ds \biggr)^\frac q2, \forall t\in [0,T], \forall q\in (1,\infty).
	\end{equation}
	
	Now, we impose the following set of conditions on the nonlinear term $G(\cdot)$ and the Wiener process $W$.
	\begin{enumerate}[label=(\textbf{N})]
		\item \label{noise}Let $\mathscr{H}$ be a separable Hilbert space.	We assume that the driving noise $W$ is a $\mathscr{H}$-valued Wiener process with a positive and symmetric covariance operator $Q\in \mathscr{L}_1(\mathscr{H})$.
	\end{enumerate}
	
	\begin{enumerate}[label=(\textbf{G})]
		\item \label{nonlin-g}	We assume that the nonlinear function $G:\h\to \mathscr{L}(\mathscr{H},\ve_\frac14) $ is measurable and that there exists a constant $C_1>0$ such that for any $\bu \in \h$, $\bv\in \h$ we have 
		\begin{equation*}
		\Vert G(\bu)-G(\bv)\Vert_{\mathscr{L}(\mathscr{H},\ve_{\frac14})}\le C_1\lvert \bu -\bv \vert.
		\end{equation*}
	\end{enumerate}
	\begin{rem}\label{REM-G}
		\begin{enumerate}[label=(\alph{*})]
			\item[]
			\item \label{REM-G_i} Note that the above assumption implies that $G:\h \to \mathscr{L}(\mathscr{H},\h)$ is globally Lipschitz and of at most linear growth, i.e, there exists a constant $C_2>0$ such that 
			\begin{align*}
			\Vert G(\bu)-G(\bv)\Vert_{\mathscr{L}(\mathscr{H},\h)} \le C_2\vert \bu-\bv \vert,\\
			\vert G(\bu) \vert \le C_2(1+ \vert \bu \vert),
			\end{align*}
			for any $\bu, \;\bv \in \h$.
			\item \label{REM-G-ii} There also  exists a number $C_3>0$ such that 
			{
				\begin{align*}
				\Vert G(\bu)-G(\bv)\Vert_{\mathscr{L}(\mathscr{H},\ve_\frac14)} &\le C_3\Vert \bu-\bv \Vert_{\frac14},
				\\
				\Vert G(\bu) \Vert_{\mathscr{L}(\mathscr{H},\ve_\frac14)} &\le C_3(1+ \Vert \bu \Vert_{\frac14}),
				\end{align*}}
			for any $\bu, \;\bv \in \ve_\frac14$.
			\item Owing to item \ref{REM-G_i} of the present remark, { if $\bu\in \mathscr{M}^2_T(\h)$, then $G(\bu)\circ Q^\frac12\in \mathscr{M}^2_T(\mathscr{L}_2(\mathscr{H},\h))$ }and the stochastic integral $\int_0^t G(\bu(s))dW(s)$ is a well defined $\h$-valued martingale.
		\end{enumerate}
	\end{rem}
	
	To close the current subsection we formulate the following remark.
	\begin{rem}\label{REM-NEW}
		Our assumptions on our problem do not imply the assumptions in either \cite{gyongy2009rate} or \cite{bessaih2007upper}. To justify this claim assume that the coefficient of the noise $G$ of our paper and those of    \cite{gyongy2009rate} and \cite{bessaih2007upper} are both zero. Let us now set $$ A(t,u)=-\rA u -B(u,u), $$ which basically corresponds to the drift in both \cite{gyongy2009rate} and \cite{bessaih2007upper}. For the sake of simplicity we take $\theta=0$ and $\gamma=\frac14$ in our assumption \ref{NB-a}. The spaces  $H$ and $V$ in \cite{gyongy2009rate} and \cite{bessaih2007upper} are respectively $\ve_0$ and $\ve_\frac12$ in our framework. The map $A(t, u)$ defined above satisfies
		\begin{equation*}
		\langle A(t,u)-A(t,v), u-v \rangle \le -\lvert u-v \rvert^2 + C_0 \lvert u-v\rvert \lVert u- v\rVert_\frac14 \left(\lVert u\rVert_{\frac14}+\lVert v \rVert_{\frac14}\right).  
		\end{equation*} 
		This implies that our assumptions does not imply either  \cite[Assumptions 2.1(i) and (2.2)(1)]{gyongy2009rate} or \cite[Assumption (H2)]{bessaih2007upper}.  
	\end{rem}

	\subsection{Preliminary results}
	In this subsection we recall and {derive some results that will be used }in the remaining part of the paper. To this end, we first define the notion of solution of \eqref{spdes}.
	\begin{defn}
		An $\mathbb{F}$-adapted process $\bu$ is called a {\textit{ weak  solution}} of \eqref{spdes} (in the sense of PDEs)
		if the following conditions are satisfied
		\begin{enumerate}[label=(\roman{*})]
			\item $\bu \in \el^2(0,T; \ve)\cap C([0,T]; \h) $ \quad $\mathbb{P}$-a.s.,
			\item for every $t\in [0,T]$ we have $\mathbb{P}\mbox{-a.s}.$
			\begin{equation}\label{Weak}
			(\bu(t),\phi)=(\bu_0,\phi)-\int_0^t\left(\langle \rA\bu(s)+\bb(\bu(s), \bu(s)),\phi\rangle \right)ds+
			\int_0^t \langle \phi , G(\bu(s)) d\bW(s)\rangle,
			\end{equation}
			{for any $\phi\in \ve$.}
		\end{enumerate}
	\end{defn}
	
	{
		\begin{defn}
			An $\mathbb{F}$-adapted process $\bu \in C([0,T]; \h) $ $\mathbb{P}$-a.s. is called a {\textit{mild solution}} to \eqref{spdes} if 
			for every $t\in [0,T]$, 
			\begin{equation}\label{Eq-REM}
			\bu(t)=e^{-t\rA} \bu_0 +\int_{0}^t e^{-(t-r)\rA} \mb(\bu(r), \bu(r)) dr + \int_0^t e^{-(t-r)\rA} G(\bu(r)) d\bW(r), \text{ $\mathbb{P}$-a.s}.
			\end{equation}
	\end{defn}}
	{
		\begin{rem}\label{REM-Mild}
			Observe that if $\bu \in \el^2(0,T;\ve)\cap C([0,T],\h)$ is a mild solution to \eqref{spdes}, then for any $t>s\ge0$,
			\begin{equation*}
			\bu(t)=e^{-(t-s)\rA} \bu(s) +\int_{s}^t e^{-(t-r)\rA} \mb(\bu(r), \bu(r)) dr + \int_s^t e^{-(t-r)\rA} G(\bu(r)) dW(r), \text{ $\mathbb{P}$-a.s}.
			\end{equation*}
			In fact, we have 
			\begin{align*}
			\bu(t)=& e^{-(t-s)\rA}\left(e^{-s\rA} \bu_0 +\int_{0}^s e^{-(s-r)\rA} \mb(\bu(r), \bu(r)) dr + \int_0^s e^{-(s-r)\rA} G(\bu(r)) dW(r)\right)\\
			&\quad + \int_{s}^t e^{-(t-r)\rA} \mb(\bu(r), \bu(r)) dr + \int_s^t e^{-(t-r)\rA} G(\bu(r)) dW(r)\\
			=& e^{-(t-s)\rA}\bu(s)+\int_{s}^t e^{-(t-r)\rA} \mb(\bu(r), \bu(r)) dr + \int_s^t e^{-(t-r)\rA} G(\bu(r)) dW(r), \text{ $\mathbb{P}$-a.s}. 
			\end{align*}
		\end{rem}
		{ This remark is used later}  to prove a very important lemma for our analysis, see Lemma \ref{LEM-Essence}.
		
		Next, we state and give a short proof of the following results.}
	\begin{prop}\label{SHELL}
		If the assumptions \ref{NB-a} to \ref{NB-c} hold and \ref{nonlin-g} is satisfied with $\ve_\frac14$ replaced by $\h$  and $\bu_0\in \el^2(\Omega,\h)$, then the problem~\eqref{spdes} has a unique global mild, which is also a weak, solution $\bu$. {Moreover}, if $\bu_0\in \el^{2p}(\Omega,\h)$ for any real number $p\in [2,8]$, then  there
		exists a constant $\mathcal{C}>0$ such that
		\begin{equation}\label{Eq-9}
		\mathbb{E} \sup_{t\in [0,T]}\lvert \bu(t)\rvert^{2p} +\mathbb{E} \int_0^T \lvert \bu(s)\rvert^{2p-2} \lvert
		\mathrm{A}^\frac12 \bu(s)\rvert^2 ds\le \mathcal{C}(1+\mathbb{E} \lvert \bu_0\rvert^{2p}),
		\end{equation}
		and 
		\begin{equation}\label{Eq-9-B}
		\mathbb{E} \left(\int_0^T\lvert
		\mathrm{A}^\frac12 \bu(s)\rvert^2 ds\right)^p \le \mathcal{C}(1+\mathbb{E} \lvert \bu_0\rvert^{2p}).
		\end{equation}
		If, in addition, Assumption \ref{nonlin-g} is satisfied and  $\bu_0 \in \el^{p}(\Omega, \ve_{\frac14})$ with $p\in [2,8]$,  then there exists a constant $C>0$ such that
		\begin{equation}\label{Eq-10}
		\mathbb{E}\sup_{t\in [0,T]} \Vert \bu(t)\Vert_{\frac14}^{p} + \EE\left(\int_{0}^{T} \Vert \bu(s)\Vert_{\frac34}^2ds\right)^p \le C(1+\EE \Vert \bu_0\Vert^{p}_{\frac14}+ (\mathbb{E} \lvert \bu_0\rvert^{2p})^2).
		\end{equation}
	\end{prop}
	\begin{proof}
		Let us first prove the existence of a local mild solution. For this purpose,  we study the properties {of} $\mb$ in order to apply a contraction principle as in \cite[Theorem 3.15]{BHR-14}.
		Let $\mb(\cdot)$ be the mapping defined by $\mb(\bx)=\mb(\bx,\bx)$ for any $\bx \in \ve_\beta$. Let {$\beta \in (0,\frac12)$}. Using  Assumptions \ref{NB-a} with $\theta= \frac12-\beta$, $\gamma=\beta$, we derive that
		\begin{equation}\label{Local-Lip}
		\Vert \mb(\bx)-\mb(\by)\Vert_{{\beta-\frac12}} \le  C_0 \vert \bx -\by \vert (\Vert \bx \Vert_\beta + \Vert \by\Vert_\beta)+ C\Vert \bx -\by \Vert_\beta (\vert \bx \vert +\vert \by \vert),
		\end{equation}
		for any $\bx, \by \in \ve_\beta$. Since, by \cite[Theorem 1.18.10, pp 141]{Triebel}, $\ve_\beta$ {coincides} with the complex interpolation $[\h, D(\rA^\frac12)]_{2\beta}$, we infer from the interpolation inequality \cite[Theorem 1.9.3, pp 59]{Triebel} and 
		\eqref{Local-Lip} that 
		\begin{equation}\label{Loc-Lip-2}
		\Vert \mb(\bx)-\mb(\by)\Vert_{{\beta-\frac12}} \le  C_0 \vert \bx -\by \vert (\vert \bx \vert^{1-2\beta}\Vert \bx \Vert_\frac12^{2\beta} + \vert \by \vert^{1-2\beta}\Vert \by\Vert_\frac12^{2\beta})
		+ C\Vert \bx -\by \Vert_\frac12^{{2\beta}}\vert \bx-\by \vert^{1-2\beta} (\vert \bx \vert +\vert \by \vert),
		\end{equation}
		for any $\bx, \by \in \ve$. Now, we denote by $X_T$ the Banach space  $C([0,T];\h)\cap \el^2(0,T;\ve)$ endowed with the norm 
		$$ \Vert \bx \Vert_{X_T}=\sup_{t\in [0,T]}\vert \bx(t)\vert+\biggl(\int_0^T \Vert \bx(t)\Vert^2_\frac12 dt \biggr)^\frac12.$$  We recall  the 
		following classical result, see \cite[Theorem 3, pp 520]{RD+JL}.
		\begin{equation}\label{Lin-Cont}
		\text{ The linear map } \Lambda : \el^2(0,T;\ve^\ast)\ni f\mapsto \bx(\cdot)=\int_0^{\cdot} e^{-(\cdot-r)\rA}f(r)dr\in X_T \text{ is continuous.}
		\end{equation}
		Thus, thanks to \eqref{Loc-Lip-2}, \eqref{Lin-Cont} and Assumption \ref{nonlin-g} we can apply \cite[Theorem 3.15]{BHR-14} to infer the existence of a unique local mild solution $\bu$ with lifespan $\tau$ of \eqref{spdes} (we refer to \cite[Definition 3.1]{BHR-14} for the definition of local solution). Let $\{\tau_j; \; j\in \mathbb{N}\}$ be an increasing sequence of stopping times converging almost surely to the lifespan $\tau$. Using the equivalence lemma in  \cite[Proposition 6.5]{DP+JZ-92} we can easily prove that the local mild solution is also a local weak solution satisfying \eqref{Weak} with $t$ replaced by $t\wedge \tau_j$, $j\in \mathbb{N}$. Now, we can prove by arguing as in \cite[Appendix A]{ZB+EM} or \cite[Proof of Theorem 4.4]{Caraballo} that the local solution $\bu$ satisfies \eqref{Eq-9} uniformly w.r.t. $j \in \mathbb{N}$. With this observation along with an argument similar to \cite[Proof of Theorem 2.10]{BHR-14} we conclude that \eqref{spdes} admits a global solution (\textit{i.e.}, $\tau=T$ a.s.) $\bu$ satisfying \eqref{Eq-9} and $\bu \in X_T$ almost-surely. 
		
		As mentioned earlier the proof follows a similar argument as in \cite[Appendix A]{ZB+EM}, but for the sake of completeness we sketch the proof of \eqref{Eq-9}.  We apply It\^o's formula first to $|\cdot|$ and the process $\bu(\cdot\wedge\tau_j)$ and then to the map $x\to x^{p}\quad p\geq 2$ and the process $|\bu(\cdot\wedge\tau_j)|^2$. Then, using the assumption {\bf (B2)} and {\bf (G)} we infer that there exists a constant $\mathcal{C}>0$ such that for any $j \in \mathbb{N}$
		%
		\begin{align*} 
		& \sup_{t\in [0,T]}\lvert \bu(t\wedge\tau_j))\rvert^{2p} +\int_0^T \lvert \bu(s)\rvert^{2p-2} \lvert
		\mathrm{A}^\frac12 \bu(s)\rvert^2 ds\le \mathcal{C}\mathbb{E} \lvert \bu_0\rvert^{2p}	
		\\
		&+ \mathcal{C} \int_0^T \lvert \bu(s\wedge\tau_j)\rvert^{2p-2} \lvert (1+  \lvert \bu(s\wedge\tau_j)\rvert^{2} )ds
		\\
		&+2p \sup_{t\in [0,T]} \int_{0}^{t\wedge\tau_j}|\bu(s)|^{2p-2}{\langle \bu(s), G(\bu(s))dW(s)\rangle }.
		\end{align*}
		Using the Burkholder-Holder-Davis inequality we deduce that
		\begin{align*} 
		\mathbb{E}   \sup_{t\in [0,T]} \int_{0}^{t\wedge\tau_j}|\bu(s)|^{2p-2}\langle G(\bu(s)), \bu(s)\rangle dW(s) \leq
		&\mathbb{E} \left( \int_{0}^{T}(|\bu(s\wedge\tau_j)|^{4p}{ds}\right)^{1/2}
		\\
		+&\mathbb{E}\left( \int_{0}^{T}(|\bu(s\wedge\tau_j)|^{4p-2}{ds}\right)^{1/2}.
		\end{align*}   
		Using Young's inequality, we infer that for any $\epsilon \in (0,\frac12)$ there exists a constant $C(\epsilon)>0$ such that
		\begin{align*}
		\mathbb{E} \left(\int_{0}^{T}(|\bu(s\wedge\tau_j)|^{4p}{ds}\right)^{1/2}
		\leq \epsilon \mathbb{E}\sup_{t\in [0,T]}\lvert \bu(t\wedge\tau_j))\rvert^{2p} +
		C(\epsilon) \mathbb{E}\int_{0}^{T}\sup_{s\in [0,t]}|\bu(s\wedge\tau_j)|^{2p}dt.
		\end{align*}
		For the second integral, we need to use  H\"older's inequality and then Young's inequality and the previous calculations 
		\begin{align*}
		\mathbb{E}\left( \int_{0}^{T}(|\bu(s\wedge\tau_j)|^{4p-2}{ds}\right)^{1/2}
		\leq \epsilon \mathbb{E}\sup_{t\in [0,T]}\lvert \bu(t\wedge\tau_j))\rvert^{2p} +
		C(\epsilon) \mathbb{E}\int_{0}^{T}\sup_{s\in [0,t]}|\bu(s\wedge\tau_j)|^{2p}dt+ T^{\frac{1}{2p}}.     
		\end{align*}
		Now collecting all the estimates we get that
		\begin{align*} 
		(1-2\epsilon)\mathbb{E} \sup_{t\in [0,T]}\lvert \bu(t\wedge\tau_j))\rvert^{2p} +\int_0^T \mathbb{E}\lvert \bu(s)\rvert^{2p-2} \lvert
		\mathrm{A}^\frac12 \bu(s)\rvert^2 ds\le &\mathcal{C}(1+\mathbb{E} \lvert \bu_0\rvert^{2p})
		\\
		&+
		\mathcal{C}\mathbb{E}\int_{0}^{T}\sup_{s\in [0,t]}|\bu(s\wedge\tau_j)|^{2p}dt.
		\end{align*} 
		Now, choosing $\epsilon=\frac14$, applying Gronwall's lemma and passing to the limit  as $j \to \infty$ complete the proof of \eqref{Eq-9}. 
		The estimate \eqref{Eq-9-B} easily follows from \eqref{Eq-9}, so we omit its proof.

		We shall now prove the inequality \eqref{Eq-10}. 
		%
		To start with we will  apply It\^o's formula to $\varphi(\bu)=\Vert \bu \Vert_{\frac14}^2$. Note that thanks to the estimates \eqref{Eq-9} and \eqref{Eq-9-B}, Assumptions \ref{NB-a} and \ref{nonlin-g} we readily check that there exists a constant $C>0$ such that $$ \mathbb{E} \int_0^T\biggl[\lVert  A\bu+B(\bu,\bu) \rVert^2_{-\frac12 } + \lVert G(\bu)\rVert^2_{\mathscr{L}(\mathscr{H}, \ve_\frac14 ) }\biggr](t)dt \le C . $$ 
		Hence the general It\^o's formula in \cite[Section 3]{Krylov-2013} is applicable to \eqref{spdes} and the functional $\varphi(\bu)(t)=\Vert \bu(t) \Vert^2_\frac14$. Thus, 
		an application of It\^o's formula to the functional $\varphi(\bu)(t\wedge \tau_j)=\Vert \bu(t\wedge \tau_j) \Vert^2_\frac14$ {gives}
		\begin{equation*}
		\varphi(\bu(t\wedge \tau_j))= \varphi(\bu(0))+ \int_{0}^{t\wedge \tau_j }\varphi\prime(\bu(s))d\bu(s)+ \frac12 \int_{0}^{t\wedge \tau_j }\trace\left(\varphi^{\prime \prime}(\bu(s)) G(\bu(s))Q(G\bu(s))^\ast\right)ds,
		\end{equation*}
		which along with the inequality $ \frac12 \Vert \phi^{\prime \prime }(\bu)\Vert\le 1$, where the norm is understood as the norm of a bilinear map, implies  
		
		\begin{equation}
		\begin{split}
		\Vert \bu(t\wedge \tau_j )\Vert^2_{\frac14}
		&+ 2\int_0^{t\wedge \tau_j}\left( \Vert \bu(s)\Vert^2_{\frac34} +2 \langle \rA^{\frac12}\bu(s), \mb(\bu(s), \bu(s))\rangle\right){ds} 
		\\
		&\le \lVert \bu_0 \rVert^2_\frac14 + 2
		\int_0^{t \wedge \tau_j }{\langle \rA^{\frac12}\bu(s), G(\bu(s)) dW(s)\rangle}+ C\trace{Q} \int_0^{t \wedge \tau_j }\Vert G(\bu(s))\Vert^2_{\cL(\mathscr{H}, \ve_\frac14)}ds .\label{Eq-11}
		\end{split}
		\end{equation}  
		Since the embedding $\ve_{\frac12+\alpha}\subset \ve_{2\alpha}$ is continuous for any $\alpha \in[0,\frac12]$, we can use Assumptions  \ref{Assum-B1p}  and the Cauchy inequality to infer that 
		\begin{equation*}
		\begin{split}
		\bigg\lvert \int_0^{t\wedge \tau_j} \langle \rA^{\frac12}\bu(s), \mb(\bu(s),\bu(s) )\rangle ds\bigg\lvert &\le C \int_0^{t\wedge \tau_j} \Vert \bu(s)\Vert_{\frac12} \vert  \mb(\bu(s), \bu(s))\vert ds,\\
		&\le \frac12 \int_0^{t\wedge \tau_j} \Vert \bu(s)\Vert_{\frac34}^2 ds +C \int_0^{t\wedge \tau_j} \Vert \bu(s)\Vert_{\frac12-\gamma}^2 \Vert \bu(s)\Vert_{\gamma}^2 ds,
		\end{split}
		\end{equation*}
		for some $\gamma\in (0,\frac12)$.  
		From an application of a complex interpolation inequality, see \cite[Theorem 1.9.3, pp 59]{Triebel},  we infer that  
		\begin{align*}
		\bigg\lvert \int_0^T \langle \rA^{\frac12}\bu(s), \mb(\bu(s),\bu(s) )\rangle ds\bigg\lvert \le& \frac12 \int_0^T \Vert \bu(s)\Vert_{\frac34}^2 ds + \int_0^T \vert \bu(s)\vert^{2}\Vert \bu(s)\Vert^{2}_\frac12 ds.
		\end{align*}
		%
		%
		Plugging the latter inequality into \eqref{Eq-11}, using the assumption on $G$ we obtain 
		\begin{equation}
		\begin{split}
		\Vert \bu(t\wedge \tau_j)\Vert^2_{\frac14}+\frac32 \int_0^{t\wedge \tau_j} \Vert \bu(s)\Vert^2_{\frac34} ds \le \Vert \bu(0)\Vert^2_{\frac14}+  C\sup_{s\in [0,T]}\vert \bu(s)\vert^2 \int_0^T\Vert \bu(s)\Vert^2_\frac12 ds\\ + CT + C\int_{0}^{T} \Vert \bu(s)\Vert^2_{\frac14} ds + 2\biggl\vert\int_0^{t\wedge \tau_j}
		{\langle \rA^{\frac14}\bu(s), \rA^\frac14 G(\bu(s)) dW(s)\rangle }\biggr\vert.
		\end{split}
		\end{equation}
		Taking the supremum over $t\in [0,T]$, then raising both sides of the resulting inequality to the power $p/2$,  taking the mathematical expectation, and finally using the Burkholder-Davis-Gundy inequality yield
		\begin{equation}\label{2-18}
		\begin{split}
		\EE\sup_{s\in [0,t]}\Vert \bu(s\wedge \tau_j)\Vert^{p}_{\frac14}+2\EE\left( \int_0^{t\wedge \tau_j} \Vert \bu(s)\Vert^2_{\frac34} ds\right)^{p/2}-\left(	C\EE \Vert \bu(0)\Vert^{p}_{\frac14}+ CT+C\EE\left[\int_0^{t\wedge \tau_j} \Vert \bu(s)\Vert^2_{\frac14}ds\right]^\frac p2\right) \\
		\quad \quad	\le  C \left(\EE\sup_{s\in [0,T]} \vert \bu(s)\vert^{2p}\right)^\frac12\biggl[ \EE \left(\int_{0}^{T} \Vert \bu(s)\Vert^2_{\frac12}  ds\right)^{p}\biggl]^\frac12 \\
		+ 2C \EE\biggl(\int_0^{t\wedge \tau_j}
		\vert \rA^{\frac14}\bu(s)\vert^2\Vert G(\bu(s))\Vert^2_{\cL(\mathscr{H}, \ve_\frac14)} ds\biggr)^\frac p4.
		\end{split}
		\end{equation}
		Here we have used the fact that for any integer $\ell$ and $n$ we can find a constant $C_{\ell,n}$ such that 
		\begin{equation}\label{Binom}
		\sum_{i=1}^n a_i^\ell \le \left(\sum_{i=1}^n a_i \right)^{\ell}\le C_{\ell,n} \sum_{i=1}^n a_i^\ell 
		\end{equation}
		for a sequence of non-negative numbers $\{a_i;\;  i=1,2,\ldots, n \}$. 
		
		Using the assumptions on $G$ and Young's inequality we infer that there exists a constant $C>0$ such that for any $j\in \mathbb{N}$
		\begin{equation*}
		\EE\biggl(\int_0^{t\wedge \tau_j}
		\vert \rA^{\frac14}\bu(s)\vert^2\Vert G(\bu(s))\Vert^2_{\cL(\mathscr{H}, \ve_\frac14)} ds\biggr)^\frac p4\le \frac14 \EE\sup_{s \in [0,t]}\lVert \bu (s\wedge \tau_j)\rVert^p_\frac14 + C \EE \left[\int_0^{t\wedge \tau_j} \lVert \bu(s) \rVert^2_\frac14  ds \right]^\frac p2 + C T,
		\end{equation*}
		which along with \eqref{2-18}, \eqref{Eq-9} and \eqref{Eq-9-B} implies
		\begin{equation*}
		\begin{split}
		\EE\sup_{s\in [0,t]}\Vert \bu(s\wedge \tau_j)\Vert^{p}_{\frac14}+2\EE\left( \int_0^{t\wedge \tau_j} \Vert \bu(s)\Vert^2_{\frac34} ds\right)^\frac p2 \le &\EE \Vert \bu(0)\Vert^{p}_{\frac14}+ \mathcal{C}^2(1+\mathbb{E} \lvert \bu_0\rvert^{2p})^2
		\\
		&+ CT+\EE\left[\int_0^{t\wedge \tau_j} \Vert \bu(s)\Vert^{2}_{\frac14}ds\right]^\frac p2.
		\end{split}
		\end{equation*}
		Now,  we infer from the interpolation inequality \cite[Theorem 1.9.3, pp 59]{Triebel}, \eqref{Eq-9} and \eqref{Eq-9-B} that there exists a constant $C>0$ such that for any $j \in \mathbb{N}$
		\begin{align*}
		\EE\left[\int_0^{t\wedge \tau_j} \Vert \bu(s)\Vert^{2}_{\frac14}ds\right]^\frac p2\le T^\frac p2\EE\left( \sup_{s\in [0,T]}\lvert \bu(s)\rvert^\frac p2 \left[ \int_0^T \lVert \bu(s)\rVert^2ds \right]^\frac p4 \right)
		\\
		\le C T.
		\end{align*}	 		
		Hence, 
		\begin{equation*}
		\EE\sup_{s\in [0,t]}\Vert \bu(s\wedge \tau_j)\Vert^{p}_{\frac14} \le 	C_T (1+\EE \Vert \bu(0)\Vert^{p}_{\frac14}+ (\mathbb{E} \lvert \bu_0\rvert^{2p})^2),
		\end{equation*}						 		
		from 	which along with a passage to the limit we readily complete the proof of {the}  proposition.
		
	\end{proof}

	\subsection{The numerical scheme and the main result} \label{Sec:Stable}
	Let $N$ be a positive integer, $\h_N\subset\h$ the linear space spanned by $\{\psi_n; \; n=1,\ldots, N\}$, and $\pi_N: \h\rightarrow \h_N$ the orthogonal projection of $\h$ onto the finite dimensional subspace $\h_N$. The projection of $\bu$ by $\pi_N$ is denoted by
	\begin{equation}
	\bu^N:= \pi_N\bu = \sum_{n=1}^N (\psi_n,\bu)\psi_n,
	\end{equation}
	for $\bu\in \h$.
	{The Galerkin approximation} of the SPDEs \eqref{spdes} reads
	\begin{equation}
	\label{semi-discretized_SDEs}
	d\bu^N = [\pi_N \rA\bu^N + \pi_N \mb(\bu^N,\bu^N)]dt + \pi_N G(\bu^N)dW(t), \quad \bu^N(0)= \pi_N \bu_0.
	\end{equation}
	Due to the assumptions \ref{NB-a}-\ref{NB-c} and \ref{nonlin-g}, we can use Proposition \ref{SHELL} to prove that \eqref{semi-discretized_SDEs} has a global weak solution. 
	
	To derive an approximation of the exact solution $\bu$ of \eqref{spdes} we construct an approximation $\uj$ of the Galerkin solution $\bu^N$. To this end, 	let $M$ be a positive integer and $I_M = \left([t_m, t_{m+1}]\right)_{m=0}^{M}$ an equidistant grid of mesh-size $k=t_{m+1}-t_m$ covering $[0,T]$.
	Now,  for any $j\in \{0,\ldots, M-1\}$  we look for a sequence of $\FF$-adapted random variables $\uj\in \h_N$, $j=0,1,\ldots, M$ such that for any $w\in \ve$ 
	
	\begin{equation}\label{Algo}
	\left\{
	\begin{split}
	&\ru^0= \pi_N \bu_0,
	\\
	& \langle \ujp-\uj+k [\pi_N \rA \ujp +\pi_N \mb(\uj, \ujp)], \;  w \rangle= \langle w,  \pi_N G(\uj)\Delta_{j+1}W\rangle,
	\end{split}
	\right.
	\end{equation}
	where $\Delta_{j+1} W:=W(t_{j+1})-W(t_{j})$, $j\in \{0,\ldots, M-1\}$, is an independent and identically distributed random variables.
	We will justify in the following proposition that for a given $U_0=\pi_N\bu_0$ the numerical scheme \eqref{Algo} admits at least one solution $U^{j}\in\h_N$,  $j\in \{1,\ldots M\}$ and that \eqref{Algo} is stable in $\h$ and $D(\rA^\frac14)$.
	\begin{prop}\label{PROP-STAB}
		Let the assumptions \ref{NB-a}-\ref{NB-c} and  \ref{nonlin-g} hold.
		Let $N$ and $M$ be two fixed positive integers and $\bu_0 \in \el^{2^p}(\Omega;\h)$ for any integer $p\in [2,4]$. Then, for any $j\in \{1, \ldots, M\}$ there exists at least a $\mathcal{F}_{t_{j}}$-measurable random variable $\uj\in \h_N$ satisfying \eqref{Algo}.
		Moreover, there exists a constant $C>0$ ({depending} only on $T$ and $\trace Q$ ) such that 
		\begin{align}
		& \EE \max_{0\le m\le M} \vert \ru^m \vert^{2}+ \sum_{j=0}^{M-1} \vert \ujp -\uj \vert^2 + 2 k \EE  \sum_{j=1}^{M} \Vert \uj \Vert^2_\frac12  \le C(\EE \vert \bu_0\vert^{2}+1),\label{EQ:Stabl1} 
		\\
		& \EE \biggl[\max_{1\le m\le M} \lvert \ru^m \rvert^{2^p} +k {\sum_{j=1}^M \lvert \ru^j\rvert^{2^{p-1}}  \rVert \ru^j \lVert^2}_\frac12 \biggr] \le C(1+ \EE \lvert \bu_0 \rvert^{2^{p-1}}),\label{STAB-l2-A} 
		\\
		& \EE\biggl[k {\sum_{j=1}^{M} \lVert \ru^j\rVert^2}_\frac12  \biggr]^{2^{p-1}}\le C(1+ \EE \lvert \bu_0\rvert^{2^p}).\label{STAB-l2-B}
		\end{align}

		Furthermore, if $\bu_0\in \el^8(\Omega, D(\rA^\frac14))$, then there exists a constant $C>0$ such that 
		\begin{align}
		\EE\max_{1\le m \le M}\Vert \ru^m \Vert^{2}_\frac14+ \EE\sum_{j=0}^{M-1} \Vert \ujp -\uj \Vert^2_\frac14 + k\EE \sum_{j=1}^M \Vert \uj \Vert^2_{\frac34} \le C ,\label{STAB-DA-1}
		\\
		\EE\max_{1\le m \le M}\Vert \ru^m \Vert^{4}_\frac14+ \EE\biggl(\sum_{j=0}^{M-1} \Vert \ujp -\uj \Vert^2_\frac14 \biggr)^2+ k^2 \EE \biggl(\sum_{j=1}^M \Vert \uj \Vert^2_{\frac34}\biggr)^2 \le C \label{STAB-DA-2}
		\end{align}
	\end{prop}
	\begin{proof}
		The detailed proofs of the existence, measurability and the estimates \eqref{STAB-DA-1} and \eqref{STAB-DA-2} will be given in Section \ref{Sec-Stabl}. Thanks to the assumption \ref{NB-b}, the proof of the inequalities \eqref{EQ:Stabl1}-\eqref{STAB-l2-B} is very similar to the proof of \cite{Carelli+Prohl}, so we omit it.
	\end{proof}
	{ We should note that the estimates \eqref{STAB-DA-1} and \eqref{STAB-DA-2} hold even if $\bu_0\in \el^4(\Omega, D(\rA^\frac14))$, but for the sake of consistency we take $\bu_0\in \el^8(\Omega, D(\rA^\frac14))$.}
	
	Now, we proceed to the statement of the main result of this paper. 
	\begin{theorem}\label{THM-MAIN}
		Let the assumptions \ref{NB-a}-\ref{NB-c} and \ref{nonlin-g} hold and assume that  $\bu_0 \in \el^{16}(\Omega; \h)\cap \el^8(\Omega;\ve_\frac14)$.  Then for any $\beta \in [0,\frac14)$, there exists a constant $\mathrm{k}_0 >0$ such that for any small number $\eps>0$ we have 
		\begin{equation}\label{RATE}
		{\max_{1\le j\le M}} \mathbb{E}\left(\mathbf{1}_{\Omega_k} \lVert \bu(t_j)-\uj \rVert^2_\beta\right)+2k \EE\left(\mathds{1}_{\Omega_k}\sum_{j=1}^M\Vert \bu(t_j)-\uj \Vert^2_{\frac12+\beta}\right)<\mathrm{k}_0 k^{-2\eps} [k^{2(\frac14-\beta)}+\lambda_N^{-2(\frac14-\beta)}],
		\end{equation}
		where the set $\Omega_k$ is defined by 
		$$\Omega_k = \left\lbrace \omega\in \Omega:\sup_{t\in [0,T]}\lVert \bu(t,\omega)\rVert^2_\frac14<\log{k^{-\eps}}, \max_{0\le j\le M} \Vert \uj(\omega )\Vert^2_{\frac14}< \log{k^{-\eps}}\right\rbrace. $$
	\end{theorem}
	\begin{proof}
		The proof of this theorem will be given in Section \ref{Sec-Proof}.
	\end{proof}
	\begin{rem}
		Note that owing to \eqref{Eq-10} and \eqref{STAB-DA-2} and the Markov inequality it is not difficult to prove that the set $\Omega_k$ satisfies
		\begin{equation*}
		\lim_{k\searrow0}\mathbb{P}[\Omega\backslash \Omega_k]=0.
		\end{equation*}
	\end{rem}
	\begin{cor}\label{COR-PROB}
		If all the assumptions of Theorem \ref{THM-MAIN} are satisfied, then the solution $\{\uj; \; j=1,\,2,\ldots, M\}$ of the numerical scheme \eqref{Algo} converges in probability in the Hilbert space $\ve_\beta$, $\beta \in [0,\frac14)$. More precisely, for any small number $\eps>0$, any $\theta_0 \in \left(0, \frac14-\beta-\eps\right)$ and $\theta_1\in (0, \frac14-\beta)$ we have 
		\begin{equation}
		\lim_{\varTheta \nearrow \infty}\lim_{k\searrow 0}\lim_{N\nearrow \infty}{\max_{1\le j\le M}}\mathbb{P}\left( \lVert \bu(t_j)-\uj \rVert_\beta+k^\frac12 \biggl(\sum_{j=1}^M\Vert \bu(t_j)-\uj \Vert^2_{\frac12+\beta}\biggr)^\frac12 \ge \varTheta[ k^{\theta_0}+\Lambda_N^{-\theta_1}] \right)=0.
		\end{equation} 
	\end{cor}
	\begin{proof}
		To shorten notation let us set $\mathbf{e}^j:= \bu(t_j)-\uj$ and
		$$ \Omega_{k,N}^\varTheta=\{\omega\in \Omega; \; \lVert \mathbf{e}^j \rVert_\beta^2+k \sum_{j=1}^M\Vert \mathbf{e}^j \Vert^2_{\frac12+\beta} \ge\varTheta[ k^{\theta_0}+\Lambda_N^{-\theta_1}]\},$$
		for any positive numbers $M$ and $k$. Let $\Omega_k$ be as in the statement of Theorem \ref{THM-MAIN}. 
		Owing to \eqref{RATE}, \eqref{Eq-10}, \eqref{STAB-DA-2} and the Chebychev-Markov inequality,  we can find a constant $\tilde{C}_5>0$ such that 
		\begin{align*}
		\mathbb{P}\left(\Omega_{k,N}^\varTheta\right)&=  \mathbb{P}(\Omega_{k,N}^\varTheta\cap \Omega_k ) +\mathbb{P}(\Omega_{k,N}^\varTheta\cap \Omega^{\mathrm{c}}_{k})
		\\
		&\le  \mathbb{P}(\Omega_{k,N}^\varTheta \cap \Omega_k ) + \mathbb{P}(\Omega^{\mathrm{c}}_{k})
		\\
		&\le  \frac{\mathrm{k}_0}{\varTheta} k^{2(\frac14-\beta)-2\eps-2\theta_0}+\frac{\mathrm{k}_0}{\varTheta} k^{-2\eps} \lambda_N^{-2(\frac14-\beta)+2\theta_1} +\frac{\tilde{C}_5}{\log{k^{-\eps}}}. 
		\end{align*}
		Letting $N\nearrow\infty$, then $k\searrow 0$ and finally $\varTheta\nearrow \infty$ in the last line we easily conclude the proof of the corollary. 
	\end{proof}
	To close this section let us make some few remarks. Instead of the scheme \eqref{Algo} we could also use a fully-implicit scheme. More precisely,  for any $j\in \{0,\ldots, M-1\}$  we look for a $\mathcal{F}_{t_{j}}$-measurable random variable $\muj\in \h_N$ such that for any $w\in \ve$ 
	{
		\begin{equation}\label{Algo-2}
		\left\{
		\begin{split}
		&\mathcal{U}^0= \pi_N \bu_0,
		\\
		& \langle \mujp-\muj+k [\pi_N \rA \mujp +\pi_N \mb(\mujp, \mujp)], \;  w \rangle= \langle w, \pi_N G(\muj)\Delta_{j+1}W\rangle,
		\end{split}
		\right.
		\end{equation}}
	where $\Delta_{j+1} W:=W(t_{j+1})-W(t_{j})$, $j\in \{0,\ldots, M-1\}$. 
	We have the following theorem:
	\begin{theorem}
		Let the assumptions \ref{NB-a}-\ref{NB-c} and \ref{nonlin-g} hold and assume that $\bu_0 \in \el^{16}(\Omega; \h)\cap \el^8(\Omega;\ve_\frac14)$. Let $N$ and $M$ be two fixed positive integers.  
		Then, 
		\begin{enumerate}[label=(\alph{*})]
			\item  for any $j\in \{0, \ldots, M-1\}$ there exists a unique  $\mathcal{F}_{t_{j}}$-measurable random variable $\muj\in \h_N$ satisfying \eqref{Algo-2} and the estimates \eqref{EQ:Stabl1} and \eqref{STAB-DA-2}.
			\item For any $\beta\in [0,\frac14)$ there exists a constant $\mathrm{k}_0 >0$ such that for any small number $\eps>0$ we have 
			{
				\begin{equation}
				{\max_{1\le j \le M}} \mathbb{E}\left(\mathbf{1}_{\Omega_k}\lVert \bu(t_j)-\muj \rVert^2_\beta\right)+2k \EE\left(\mathds{1}_{\Omega_k}\sum_{j=1}^M\Vert \bu(t_j)-\muj \Vert^2_{\frac12+\beta}\right)<\mathrm{k}_0 k^{-2\eps} [k^{2(\frac14-\beta)}+\lambda_N^{-2(\frac14-\beta)}],
				\end{equation}
			}
			where $$\Omega_k = \left\lbrace \omega:\sup_{t\in [0,T]}\lVert \bu(t,\omega)\rVert^2_\frac14<\log{k^{-\eps}}, \max_{0\le j\le M} \Vert \muj(\omega )\Vert^2_{\frac14}< \log{k^{-\eps}}\right\rbrace. $$
			\item Moreover, for any small number $\eps>0$, any $\theta_0 \in \left(0, \frac14-\beta-\eps\right)$ and $\theta_1\in (0, \frac14-\beta)$ 
			\begin{equation}
			\lim_{\varTheta \nearrow \infty}\lim_{k\searrow 0}\lim_{N\nearrow \infty} {\max_{1\le j \le M}}\mathbb{P}\left( \lVert \bu(t_j)-\muj \rVert^2_\beta+k^\frac12 \left(\sum_{j=1}^M\Vert \bu(t_j)-\muj \Vert^2_{\frac12+\beta} \right)^\frac12\ge  \varTheta[ k^{\theta_0}+{\lambda_N^{-\theta_1}}] \right)=0.
			\end{equation}
			
		\end{enumerate}
		%
		%
		%
	\end{theorem}
	\begin{proof}
		The arguments for the proof of this theorem are very similar to those of the proofs of Proposition \ref{PROP-STAB}, Theorem \ref{THM-MAIN} and Corollary \ref{COR-PROB}, thus we omit them.
	\end{proof}	
	
	\section{Existence and stability analysis of the scheme: Proof of Proposition \ref{PROP-STAB}}\label{Sec-Stabl}
	In this section we will show that for any $j\in \{0,\ldots, M-1\}$ the numerical scheme \eqref{Algo} admits at least one solution $\uj\in \h_N$. We will also show that \eqref{Algo} is stable in $D(\rA^\frac14)$, see {Proposition  \ref{PROP-STAB} for more precision.} 
	
	\begin{proof}[Proof of Proposition \ref{PROP-STAB}]
		As we mentioned in Subsection \ref{Sec:Stable} we will only prove the existence, measurability and the estimates \eqref{STAB-DA-1} and \eqref{STAB-DA-2}. The proof of the inequalities \eqref{EQ:Stabl1}-\eqref{STAB-l2-B} will be omitted because it is very similar to the proof of \cite{Carelli+Prohl} (see also \cite{ZB+EC+AP-IMA}).

		\noindent \textit{Proof of the existence.} We first establish that for any $j\in \{0,\ldots M-1\}$ there exists $\uj\in \h_N$ satisfying the numerical scheme \eqref{Algo}. To this end, let us fix $\omega\in \Omega$ and for a given $\uj\in \h_N$ consider the map $\Lambda^j_\omega: \h_N \to \h_N$ defined by 
		\begin{equation*}
		{\langle\Lambda^j_\omega(\bv), \psi\rangle= \langle\bv-\uj(\omega), \psi\rangle+ k\langle\rA \bv+ \pi_N\mb(\uj(\omega),\bv), \psi\rangle- \langle \psi,  \pi_N G(\uj(\omega))\Delta_{j+1}W(\omega) \rangle }
		\end{equation*}
		for any $\psi\in \h_N$. Note that since $\h_N\subset D(\rA)$ the map $\Lambda^j_\omega$ is well-defined. From assumptions \ref{NB-a} and \ref{nonlin-g} and the linearity of $\rA$ it is clear that for given $\uj$ the map $\Lambda^j_\omega$ is continuous. Furthermore, using H\"older's inequality, the fact that $\lambda_1\lvert \psi \vert^2 \le \Vert \psi \Vert^2_\frac12$, $\psi \in \ve$  and assumptions \ref{NB-b} and \ref{nonlin-g} we derive that 
		\begin{equation*}
		\begin{split}
		{\langle\Lambda_\omega^j\bv, \bv\rangle}\ge &\lvert\bv \vert^2 \left(\lambda_1 k + \frac12-\frac k2\right)-\frac{\vert \uj(\omega)\vert^2}{2}\left(1+\Vert \Delta_{j+1}W(\omega)\Vert^2_{\mathscr{H}}C^2_2\right)-\frac12 \Vert \Delta_{j+1}W(\omega)\Vert^2_{\mathscr{H}}C^2_2\\
		\ge & \gamma \vert \bv \vert^2 -\Gamma^j_\omega.
		\end{split}
		\end{equation*} 
		Since  $k<1$, and by Assumption \ref{noise}, $\Vert \Delta_{j+1}W\Vert^2_{\mathscr{H}}<\infty$, the constant $\gamma$ is positive and $\mu_j= \sqrt{\frac{\Gamma^j_\omega}{\gamma}}<\infty$ whenever $\vert \uj\vert^2 <\infty$. Thus, we have ${\langle\Lambda_\omega^j\bv, \bv\rangle}\ge 0$ for any $\bv \in   \mathcal{H}^j_N(\omega):=\{\psi \in \h_n;\; \vert \psi \vert = R \mu_j \}$ where $R> 1$ is an arbitrary constant.  
		Since $\ru^0=\pi_N \bu_0$ is given, we can conclude from the above observations and Brouwer fixed point theorem that there exists at least one $\ru^1\in \h_N$ satisfying  $$ 
		\Lambda^0_\omega(\ru^1)=0 \text{ and } \lvert \ru^1 \vert\le R \mu_0 .$$ In a similar way, assuming that $\uj\in \h_N$, we infer that there exists at least one $\ujp \in \h_N$ such that  $$ 
		\Lambda^{j}_\omega(\ujp)=0 \text{ and } \lvert \ujp \vert\le R \mu_j .$$
		Therefore, we have {to} prove by induction that given $\ru^0\in \h_N$ and a $\mathscr{H}$-valued Wiener process $W$, for each $j$, there exists a sequence $\{\uj; \; j=1,\ldots, M\}\subset\h_N $ satisfying the algorithm \eqref{Algo}. 
		
		\noindent \textit{Proof of the measurability.} In order to prove the $\mathcal{F}_{t_j}$-measurability of $\uj$ it is sufficient to show that for each $j\in \{1,\ldots, M\}$ one can find a Borel measurable map $\mathscr{E}_j:\h_N \times \mathscr{H}\to \h_N$ such that $\uj= \mathscr{E}_j(\ru^{j-1}, \Delta_{j}W)$. In fact, if such claim is true then by exploiting the $\mathcal{F}_{t_j}$-measurability of $\Delta_jW$ one can argue by induction and show that if $\ru^0$ is $\mathcal{F}_0$-measurable then $ \mathscr{E}_j(\ru^{j-1}, \Delta_{j}W)$ { is $\mathcal{F}_{t_j}$-measurable}, hence $\uj$ is $\mathcal{F}_{t_j}$-measurable. Thus, it remains to prove  the existence of $\mathscr{E}_j$. For this purpose we will closely follow \cite{DB+D}. Let $\mathcal{P}(\h_N) $ be the set of subsets of $\h_N$ and consider a multivalued map $\mathscr{E}^S_{j+1}: \h_N \times \mathscr{H} \to \mathcal{P}(\h_N)$ such that for each $(\uj, \eta_{j+1})$, $\mathscr{E}_{j+1}^S(\uj, \eta_{j+1})$ denotes the set of solutions $\ujp$ of \eqref{Algo}. From the existence result {above we} deduce that    $\mathscr{E}_{j+1}^S$ maps  $\h_N \times \mathscr{H}$ to nonempty closed subsets of $\h_N$. Furthermore, since we are in the finite dimensional space $\h_N$, we can prove, by using the assumptions \ref{NB-a} and \ref{nonlin-g} and the sequential characterization of the closed graph theorem, that  the graph of $\mathscr{E}_{j+1}^S$ is closed. From these last two facts and \cite[Theorem 3.1]{AB+RT} we can find a univocal map 
		$\mathscr{E}_{j+1}: \h_N \times \mathscr{H}\to \h_N$ such that $\mathscr{E}_j(\uj, \eta_{j+1})\in \mathscr{E}^S_{j+1}(\uj, \eta_{j+1})$ and $\mathscr{E}_j$ is measurable when $\h_N \times \mathscr{H}$ and $\h_N$ are equipped with their respective Borel $\sigma$-algebra.  This completes the proof of the measurability of the solutions of \eqref{Algo}.    
		
		\textit{Proof of \eqref{EQ:Stabl1}-\eqref{STAB-l2-B}.} Thanks to the assumption \ref{NB-b}, the proof of the inequalities \eqref{EQ:Stabl1}-\eqref{STAB-l2-B} is very similar to the proof of \cite{Carelli+Prohl}, so we omit it and we directly proceed to the proof of the estimates  \eqref{STAB-DA-1} and \eqref{STAB-DA-2}.

		\textit{Proof of \eqref{STAB-DA-1}.} 
		Taking $w=2\rA^{\frac12}\ujp$ in \eqref{Algo}, using the Cauchy-Schwarz inequality and the identity 
		\begin{equation}\label{ID-T}
		(( \bv-\bx, 2 \bv))= \Vert \bv \Vert^2 -\Vert \bx \rVert^2 + \Vert \bv-\bx\Vert^2, \text{ ($\bv,\bx$ \text{ are elements of a Hilbert space with norm $\lVert \cdot\rVert$ })}
		\end{equation}
		yield
		\begin{equation*}
		\begin{split}
		&\Vert \ujp \Vert^2_\frac14 -\Vert \uj \Vert^2_\frac14 + \Vert \ujp -\uj \Vert^2_\frac14 + 2 k\Vert \ujp \Vert^2_{\frac34}\\
		&\quad \le 2k \lvert {\pi_N}\mb(\uj, \ujp)\vert\Vert \ujp \Vert_{\frac12}
		+2  \lVert {\pi_N}G(\uj) \Delta_{j+1} W\Vert_\frac14 \lVert \ujp -\uj \rVert_{\frac14}\\
		&\qquad + 2 \langle \rA^\frac14 \uj ,\rA^\frac14 {\pi_N}G(\uj) \Delta_{j+1} W\rangle.
		\end{split}
		\end{equation*}
		{Using the fact that $\lVert \pi_N\rVert_{\mathscr{L}(\h, \h_N) }\le 1$, we obtain}
		\begin{equation}\label{EST-DA-quart}
		\begin{split}
		&\Vert \ujp \Vert^2_\frac14 -\Vert \uj \Vert^2_\frac14 + \Vert \ujp -\uj \Vert^2_\frac14 + 2 k\Vert \ujp \Vert^2_{\frac34}
		\\
		&\leq  2k \lvert \mb(\uj, \ujp)\vert\Vert \ujp \Vert_{\frac12}
		+2  \lVert G(\uj) \Delta_{j+1} W\Vert_\frac14 \lVert \ujp -\uj \rVert_{\frac14}
		\\
		&\qquad + 2 \langle \rA^\frac14 \uj , \rA^\frac14 {\pi_N}G(\uj) \Delta_{j+1} W\rangle.
		\end{split}
		\end{equation}
		{Using  Assumption \ref{Assum-B1p}, the complex interpolation inequality in \cite[Theorem 1.9.3, pp 59]{Triebel}, the Young inequality,  and the continuous embedding $\ve_\frac12\subset \ve_\frac14$ we obtain
			\begin{align}
			2 \lvert \mb(\uj, \ujp)\vert\Vert \ujp \Vert_{\frac12}& \le C \lvert \uj \rvert^4 \lVert \ujp \rVert^2_\frac14 + \lVert \ujp \rVert^2_{\frac34} \label{Est-nonlinB-H2}\\
			& \le C \lvert \uj \rvert^4 \lVert \ujp \rVert^2_\frac12 + \lVert \ujp \rVert^2_{\frac34}, \nonumber
			\end{align}  
			which implies that}
		\begin{equation}\label{EQ:proof-1}
		\begin{split}
		\Vert \ujp \Vert^2_\frac14 -\Vert \uj \Vert^2_\frac14 + \frac12 \Vert \ujp -\uj \Vert^2_\frac14 + 2 k \Vert \ujp \Vert^2_{\frac34} \le 2 C k  \vert \uj\vert^4 \Vert \ujp\Vert^2_{\frac12} +{4\lVert G(\uj) \Delta_{j+1} W\Vert_\frac14^2}
		\\
		+ 2  \langle \rA^\frac14 \uj , \rA^\frac14 {\pi_N} G(\uj) \Delta_{j+1} W \rangle.
		\end{split}
		\end{equation}
		Since $\ru^j$ is a constant, adapted and hence progressively measurable process, it is not difficult to prove that
		\begin{equation*}
		2 \EE \langle \rA^\frac14 \ru^j, \rA^\frac14 {\pi_N}G(\ru^j )\Delta_{j+1}W \rangle=0. 
		\end{equation*}
		Using \eqref{STAB-l2-A}  and  \eqref{STAB-l2-B} with $p=2$ and $p=3$ respectively, we easily prove that there exists a constant $C>0$, depending only on $T$,   such that
		\begin{equation}\label{nonlin-B-Est}
		k \EE \left(\sum_{j=0}^{M-1}\vert \ru^j\vert^4 \Vert \ru^{j+1} \Vert^2_\frac12\right) \le\left( \EE \max_{1\le m \le M }\vert \ru^m \vert^8 \right)^\frac12 \biggl(\EE \biggl(k \sum_{j=1}^{M} \Vert \ru^j \Vert^2_\frac12 \biggr)^2\biggr)^\frac12\le C (1+\EE \lvert \bu_0\rvert^8 )^2.
		\end{equation}
		Now, since $\uj$ is $\mathcal{F}_{t_j}$-measurable and $\Delta_{j+1}W$ is independent of $\mathcal{F}_{t_j}$, we  infer that there exists a constant $C>0$ such that for any {$j\in \{0,\ldots, M-1\} $}
		\begin{align}
		\EE\left( \lVert G(\uj) \Delta_{j+1} W\Vert^{2}_\frac14 \right)&\le \EE\left( \EE \left(\Vert G(\uj)\Vert^{2}_{\mathscr{L}(\mathscr{H},\ve_\frac14) } \Vert \Delta_{j+1} W\Vert^{2}_{\mathscr{H}}\vert \mathcal{F}_{t_j} \right) \right)\nonumber 
		\\
		&=m \EE\left( \Vert G(\uj)\Vert^{2}_{\mathscr{L}(\mathscr{H},\ve_\frac14) }  \EE \left(\Vert \Delta_{j+1} W\Vert^{2}_{\mathscr{H}}\vert \mathcal{F}_{t_j} \right) \right)\nonumber 
		\\
		&\le C k \;\left(\mathrm{tr}Q\right)^\frac12 \; (1+\EE \Vert \uj \Vert_\frac14^{2}),\label{Est-Stoc}
		\end{align}
		where  \eqref{Increment} and Assumption \ref{nonlin-g} along with Remark \ref{REM-G}-\ref{REM-G-ii}   were used to derive the last line of the above chain of inequalities.
		
		Now taking the mathematical expectation in \eqref{EQ:proof-1}, summing both sides of the resulting equations from {$j=0$ to $m-1$} and using  the last three observations imply
		\begin{equation*}
		\begin{split}
		{\max_{1\le m \le M} \EE \Vert \ru^m\Vert^2_\frac14} + \frac12 \EE \biggl(\sum_{j=0}^{M-1} \Vert \ru^{j+1} -\ru^j \Vert^2_\frac14 \biggr) + 2 k \EE\sum_{j=1}^M\Vert \ru^j \Vert^2_{\frac34}\\ \le C T + \EE \Vert \bu_0\Vert^2_\frac14 + C \trace Q k \sum_{m=1}^{M} \max_{1\le j\le m} \EE  \Vert \ru^j \Vert^2_\frac14,
		\end{split}
		\end{equation*}  
		from which along with the discrete Gronwall lemma we infer that there exists a constant $C>0$ such that 
		\begin{equation}\label{EQ:proof-2}
		\begin{split}
		\max_{1\le m \le M} \EE \Vert \ru^m\Vert^2_\frac14 + \frac12 \EE \biggl(\sum_{j=0}^{M-1} \Vert \ru^{j+1} -\ru^j \Vert^2_\frac14 \biggr) + 2 k \EE\sum_{j=1}^M\Vert \ru^j \Vert^2_{\frac34} \le  C(1+\EE \lVert \bu_0\rVert^2_\frac14 + [\EE \lvert \bu_0\rvert^8]^2).
		\end{split}
		\end{equation}  
		
		Note that from \eqref{EQ:proof-1} we can derive that there exists a constant $C>0$ such that
		\begin{equation*}
		\begin{split}
		\EE \max_{1\le m\le M}\lVert \ru^m \rVert^2_\frac14 \le\EE \Vert \bu_0 \Vert^2_\frac14 +  C k \EE \sum_{j=0}^{M-1} \lvert \uj \rvert^4 \lVert \ujp \rVert^2_\frac12 + \EE \sum_{j=0}^{M-1} \lVert G(\uj)\Delta_{j+1}W\rVert^2_\frac14 \\+ 2\EE \max_{1\le m\le M}\sum_{j=0}^{m-1}\langle \rA^\frac14 {\pi_N}G(\uj)\Delta_{j+1}W, \rA^\frac14 \uj\rangle\\=:\sum_{i=1}^4 \mathrm{I}_i.
		\end{split}
		\end{equation*}
		Arguing as in \cite[proof of (3.9)]{ZB+EC+AP-IMA} we can establish that 
		
		\begin{equation*}
		\mathrm{I}_4 \le \frac12 \EE \lVert \bu_0 \rVert^2_\frac14 +\frac12 \EE \max_{1\le m \le M} \lVert \ru^m \rVert^2_\frac14 + C k \sum_{j=0}^{M-1} \EE \lVert \uj \rVert^2_\frac14, 
		\end{equation*}
		which altogether with \eqref{EQ:proof-2} yields that 
		\begin{equation*}
		\mathrm{I}_4 \le \frac12 \EE \max_{1\le m \le M} \lVert \ru^m \rVert^2_\frac14 + C (1+ \EE \lVert \bu_0\rVert^2_\frac14).
		\end{equation*}
		Using the same idea as in the proof of \eqref{Est-Stoc} and using \eqref{EQ:proof-2} we infer that 
		\begin{equation*}
		\mathrm{I}_3\le C (1+ \EE \lVert \bu_0\rVert^2_\frac14).
		\end{equation*}
		Using these two estimates and the inequality  \eqref{nonlin-B-Est}  we derive that  there exists a constant $C>0$ such that 
		\begin{equation*}
		\EE \max_{1\le m\le M}\lVert \ru^m \rVert^2_\frac14 \le C(1+\EE \lVert \bu_0\rVert^2_\frac14 + [\EE \lvert \bu_0\rvert^8]^2),
		\end{equation*}
		which along with \eqref{EQ:proof-2} completes the proof of \eqref{STAB-DA-1}.
		
		Now, we continue with the derivation of an estimate of {$\max_{1\le m \le M}  \EE \lVert \ru^m \rVert^4_\frac14.$} Multiplying \eqref{EST-DA-quart} by $\lVert \ujp \rVert^2_\frac14$ and using identity \eqref{ID-T} and then summing both sides of the resulting equation  from {$j=0$ to $m-1$} implies
		\begin{equation}\label{Est-DA-quart-1}
		\begin{split}
		\frac12 \lVert \ru^m \rVert^4_\frac14+\frac12 \sum_{j=0}^{m-1} \biggl \lvert \lVert \ujp \rVert^2_\frac14 - \lVert \uj \rVert^2_\frac14 \biggr\rvert^2 + \sum_{j=0}^{m-1} \lVert \ujp \rVert^2_\frac14 \lVert \ujp -\uj \rVert^2_\frac14 + 2k \sum_{j=0}^{m-1} \lVert \ujp \rVert^2_\frac14 \lVert \ujp \rVert^2_\frac34
		\\
		\le \frac12 \lVert \bu_0 \rVert^4_\frac14+   C k \sum_{j=0}^{m-1} \lvert \mb(\uj,\ujp) \rvert^2 \lVert \ujp \rVert^2_\frac12 \lVert \ujp \rVert^2_\frac14\\+ 2 \sum_{j=0}^{m-1}\langle \rA^\frac14 [\ujp-\uj] , \rA^\frac14 {\pi_N}G(\uj)\Delta_{j+1}W\rangle \lVert \ujp\rVert^2_\frac14 \\ +    2 \sum_{j=0}^{m-1}\langle \rA^\frac14 \uj ,  \rA^\frac14 {\pi_N}G(\uj)\Delta_{j+1}W \rangle \lVert \ujp\rVert^2_\frac14 \\
		=:  \frac12 \lVert \bu_0 \rVert^4_\frac14+ \mathrm{J}_1 +\mathrm{J}_2 +\mathrm{J}_3.
		\end{split}
		\end{equation}
		Thanks to the estimate \eqref{Est-nonlinB-H2}  we can estimate $\mathrm{J}_1$ as follows
		\begin{equation*}
		\EE \mathrm{J}_1 \le CK \EE \sum_{j=0}^{M-1} \lvert \uj \rvert^4 \lVert \ujp \rVert^4_\frac14 + k \EE \sum_{j=0}^{M-1} \lVert \ujp \rVert^2_\frac14 \lVert \ujp \rVert^2_\frac34=: \mathrm{J}_{1,1}+\mathrm{J}_{1,2}.
		\end{equation*}
		Since  the second term $\mathrm{J}_{1,2}$ can be absorbed in the LHS later on, we will focus on estimating the second term $\mathrm{J}_{1,1}$. We have 
		\begin{equation*}
		\begin{split}
		\mathrm{J}_{1,1}\le & C k \sum_{j=0}^{M-1} \lvert \uj \rvert^4 \lvert \ujp \rvert^2 \lVert \ujp \rVert_\frac12^2
		\\
		\le&  C \biggl(\EE\max_{0\le j \le M-1}[\lvert \uj \rvert^8 \lvert \ujp \rvert^4] \biggr)^\frac12 \biggl(\EE\biggl[ k\sum_{j=1}^{M} \lVert \uj \rVert^2_\frac12  \biggr]^2 \biggr)^\frac12\\
		\le &  C \biggl(\EE[\max_{0\le j \le M-1}\lvert \uj \rvert^{12} ] \biggr)^\frac12 \biggl(\EE\biggl[ k\sum_{j=1}^{M} \lVert \uj \rVert^2_\frac12  \biggr]^2 \biggr)^\frac12\\
		\le & C (1+\EE \lvert \bu_0 \rvert^{16} ),
		\end{split}
		\end{equation*}
		where \eqref{STAB-l2-A} and \eqref{STAB-l2-B} are used to obtain the last line. Hence,
		\begin{equation*}
		\EE \mathrm{J}_1 \le  C (1+\EE \lvert \bu_0 \rvert^{16} ) + \EE k \sum_{j=0}^{M-1} \left(\lVert \ujp \rVert^2-\frac14 \lVert \ujp \rVert^2_\frac34\right).
		\end{equation*}
		Now we estimate $\mathrm{J}_2$ as follows
		\begin{equation*}
		\begin{split}
		\EE \mathrm{J}_2 \le & C \EE \sum_{j=0}^{M-1} \lVert G(\uj)\Delta_{j+1} W \rVert^2_\frac14 \left(\lVert \ujp \rVert^2_\frac14 - \lVert \uj \rVert^2_\frac14  + \lVert \uj \rVert^2_\frac14 \right) + \frac12 \EE \sum_{j=0}^{M-1} \lVert \ujp -\uj \rVert^2_\frac14 \lVert \ujp \rVert^2_\frac14 \\
		\le &  C \EE \sum_{j=0}^{M-1} \lVert G(\uj)\Delta_{j+1} W \rVert^4_\frac14 + C \EE \sum_{j=0}^{M-1} \lVert G(\uj)\Delta_{j+1} W \rVert^2_\frac14 \lVert \uj \rVert^2_\frac14 + \frac18  \EE \sum_{j=0}^{M-1}\biggl\vert \lVert\ujp\rVert^2_\frac14 -\lVert \uj \rVert^2_\frac14 \biggr\vert^2 \\ 
		& \qquad \qquad + \frac12  \EE \sum_{j=0}^{M-1} \lVert \ujp -\uj \rVert^2_\frac14 \lVert \ujp \rVert^2_\frac14.
		\end{split}
		\end{equation*}
		As long as $\mathrm{J}_3$ is concerned we have 
		\begin{equation*}
		\begin{split}
		\EE \mathrm{J}_3&= 2 \EE \sum_{j=0}^{m-1} \langle \rA^\frac14 \uj , \rA^\frac14 {\pi_N}G(\uj)\Delta_{j+1}W \rangle \lVert \uj \rVert^2_\frac14 + 2 \EE\sum_{j=0}^{m-1} \langle  \rA^\frac14 \uj, \rA^\frac14 G(\uj)\Delta_{j+1}W  \rangle \left({\lVert\ujp \rVert^2_\frac14} - \lVert \uj \rVert^2_\frac14 \right)
		\\
		&=2 \EE\sum_{j=0}^{m-1} \langle \rA^\frac14 \uj,  \rA^\frac14 {\pi_N}G(\uj)\Delta_{j+1}W \rangle \left({\lVert\ujp \rVert^2_\frac14} - \lVert \uj \rVert^2_\frac14 \right)
		\\
		&\le C \EE\sum_{j=0}^{M-1}\lVert \rA^\frac14 G(\uj)\Delta_{j+1}W\lVert^2_\frac14 \lVert \uj \rVert^2_\frac14 + \frac18 \EE\sum_{j=0}^{M-1} \biggl \lvert {\lVert\ujp \rVert^2_\frac14} - \lVert \uj \rVert^2_\frac14 \biggr \rvert^2 
		\end{split}
		\end{equation*}
		because for any $j$ $$ \EE \langle \rA^\frac14 \uj , \rA^\frac14 {\pi_N} G(\uj)\Delta_{j+1}W \rangle \lVert \uj \rVert^2_\frac14=0 .$$
		
		By a similar idea as used to derive  \eqref{Est-Stoc} we can prove that 
		\begin{equation*}
		C \EE \sum_{j=0}^{M-1} \lVert G(\uj)\Delta_{j+1} W \rVert^4_\frac14 + C \EE \sum_{j=0}^{M-1} \lVert G(\uj)\Delta_{j+1} W \rVert^2_\frac14 \lVert \uj \rVert^2_\frac14 \le C + C k \EE \sum_{j=0}^{M-1} \lVert \uj \rVert^4_\frac14.
		\end{equation*}
		Thus, 
		\begin{equation*}
		\EE[ \mathrm{J}_2 + \mathrm{J}_3] \le C + C k \EE \sum_{j=0}^{M-1} \lVert \uj \rVert^4_\frac14 +\frac14 \EE\sum_{j=0}^{M-1} \biggl \lvert \lVert\ujp \rVert^2_\frac14 - \lVert \uj \rVert^2_\frac14 \biggr \rvert^2  +\frac12 \EE \sum_{j=0}^{M-1} \lVert \ujp -\uj \rVert^2_\frac14 \lVert \ujp \rVert^2_\frac14. 
		\end{equation*}
		Taking the mathematical expectation in \eqref{Est-DA-quart-1} and by plugging the information about $\mathrm{J}_i$, $i=1,2,3$ in the resulting equation yield  
		\begin{equation*}
		\begin{split}
		{\max_{1\le m\le M}\frac12  \EE\lVert \ru^m\rVert^4_\frac14} + \frac14 \EE \sum_{j=0}^{M-1} \biggl \lvert \lVert \ujp \rVert^2_\frac14 - \lVert \uj \rVert^2_\frac14 \biggr\rvert^2 
		\\
		+ \frac12 \EE \sum_{j=0}^{M-1} \lVert \ujp \rVert^2_\frac14 \lVert \ujp -\uj \rVert^2_\frac14 + k \EE \sum_{j=1}^{M} \lVert \uj \rVert^2_\frac14 \lVert \uj \rVert^2_\frac34
		\\
		\le    C (1+\EE \lvert \bu_0 \rvert^{12} + \lVert \bu_0\rVert^4_\frac14  )+  C k \EE \sum_{j=0}^{M-1} \lVert \uj \rVert^4_\frac14, 
		\end{split}
		\end{equation*}  
		which along with the Gronwall inequality yields
		\begin{equation*}
		\begin{split}
		{\max_{1\le m\le M}\frac12  \EE\lVert \ru^m\rVert^4_\frac14} \le  C (1+\EE \lvert \bu_0 \rvert^{12} + \lVert \bu_0\rVert^4_\frac14  ).
		\end{split}
		\end{equation*}  
		The latter inequality is used in the former one to derive that 
		\begin{equation}\label{Est-DA-quart-2}
		\begin{split}
		{\max_{1\le m\le M}\frac12  \EE\lVert \ru^m\rVert^4_\frac14} + \frac14 \EE \sum_{j=0}^{M-1} \biggl \lvert \lVert \ujp \rVert^2_\frac14 - \lVert \uj \rVert^2_\frac14 \biggr\rvert^2 + \frac12 \EE \sum_{j=0}^{M-1} \lVert \ujp \rVert^2_\frac14 \lVert \ujp -\uj \rVert^2_\frac14\\ + k \EE \sum_{j=1}^{M} \lVert \uj \rVert^2_\frac14 \lVert \uj \rVert^2_\frac34
		\le   C (1+\EE \lvert \bu_0 \rvert^{12} + \lVert \bu_0\rVert^4_\frac14  ).
		\end{split}
		\end{equation} 
		Now we continue our analysis with the estimation of $\EE\max_{1\le j\le M} \lVert \uj \rVert^4_\frac14.$  To start with this analysis, we easily  derive from \eqref{Est-DA-quart-1} the following inequality
		\begin{equation*}
		\begin{split}
		{\max_{1\le m\le M}\frac12  \EE\lVert \ru^m\rVert^4_\frac14}  &\le C k \sum_{j=0}^{M-1} \lvert \uj \rvert^4 \lVert \ujp \rVert^2 \lVert \ujp \rVert^2_\frac12 
		\\
		&+ C\sum_{j=0}^{M-1} \left( \lVert G(\uj)\Delta_{j+1} \rVert^4_\frac14 +  \lVert G(\uj)\Delta_{j+1} \rVert^2_\frac14 \lVert \uj \rVert^2_\frac14 \right)
		\\
		& + \max_{0\le j \le M-1 }\sum_{\ell=0}^{j-1} \langle  \rA^\frac14 \mathrm{U}^\ell, \rA^\frac14 {\pi_N} G(\mathrm{U}^\ell)\Delta_{\ell+1}W \rangle \lVert \mathrm{U}^\ell \rVert^2_\frac14=: J_1+J_2+J_3.
		\end{split}
		\end{equation*}
		Arguing as in the proof of \eqref{Est-Stoc} and using \eqref{Est-DA-quart-2}, the mathematical expectation of $J_1+J_2$ can be estimated as follows
		\begin{equation*}
		\EE(J_1+J_2)\le C \EE(1+ \lvert \bu_0 \rvert^{16} + \Vert \bu_0 \Vert^4_\frac14). 
		\end{equation*} 
		The same idea as used in the proof of \cite[inequality (3.15)]{ZB+EC+AP-IMA} yields
		\begin{equation*}
		\EE J_3 \le \frac14 {\EE \max_{1\le m\le M} \lVert \ru^m \rVert^4_\frac14 }+ C\EE\lVert \bu_0 \rVert^4_\frac14 + C k \EE\sum_{j=0}^{M-1} \lVert \uj \rVert^4_\frac14, 
		\end{equation*}
		from which altogether with \eqref{Est-DA-quart-2} we infer that 
		\begin{equation*}
		\EE J_3 \le C \EE(1+ \lvert \bu_0 \rvert^{16} + \Vert \bu_0 \Vert^4_\frac14) + \frac14  {\EE \max_{1\le m\le M} \lVert \ru^m \rVert^4_\frac14 }. 
		\end{equation*}
		Thus, summing up we have shown that there exists a constant $C>0$ such that
		\begin{equation}\label{STAB-DA-max}
		\begin{split}
		{\EE \max_{1\le m\le M} \lVert \ru^m \rVert^4_\frac14 } \le C\EE(1+ \lvert \bu_0 \rvert^{16} + \Vert \bu_0 \Vert^4_\frac14).
		\end{split}
		\end{equation}
		
		Now, we estimate $\EE\left(\sum_{j=0}^{M-1}\lVert \ujp -\uj \rVert^2_\frac14 \right)^2 + \EE\left(k\sum_{j=1}^{M}\lVert \uj\rVert^2_\frac34 \right)^2$. To do this we first observe that from \eqref{EQ:proof-1} we infer that
		\begin{equation}
		\begin{split}
		\biggl(\frac12 \sum_{j=0}^{M-1}\lVert \ujp -\uj \rVert^2_\frac14 \biggr)^2+ \biggl(2k \sum_{j=0}^{M-1}\lVert \ujp  \rVert^2_\frac14\biggr)^2\le C \biggl(k\sum_{j=0}^{M-1} \lvert \uj \rvert^4 \lVert \ujp \rVert^2_\frac12\biggr)^2 \\+ C \biggl(\sum_{j=0}^{M-1} \lVert G(\uj)\Delta_{j+1}W \rVert^2_\frac14\biggr)^2 + C\biggl( \sum_{j=0}^{M-1} \langle \rA^\frac14 \uj , \rA^\frac14 {\pi_N}G(\uj)\Delta_{j+1}W\rangle\biggr)^2.
		\end{split}
		\end{equation} 
		Then, using the same strategies to estimate the $J_i$-s (or $\mathrm{J}_i$ ), the sum of the three terms in the right hand side of the above quality can be bounded from above by 
		\begin{equation*}
		\biggl[\EE \left(\max_{0\le j\le M} \lvert \uj \rvert^{16}\right)\biggr]^\frac12 \biggl[\EE \left(k\sum_{j=1}^M \lVert \uj \rVert^2_\frac12 \right)^4\biggr]^\frac12+ 
		C M k^2 \sum_{j=0}^M \EE \lVert \uj \rVert^4_\frac14 + C k \sum_{j=0}^M \EE \lVert \uj \rVert^4_\frac14,  
		\end{equation*}
		which along with the estimate for  ${\EE \max_{1\le m\le M} \lVert \ru^m \rVert^4_\frac14 }$ and the inequalities \eqref{STAB-l2-A} and \eqref{STAB-l2-B}  implies that
		\begin{equation}
		\begin{split}
		\biggl(\frac12 \sum_{j=0}^{M-1}\lVert \ujp -\uj \rVert^2_\frac14 \biggr)^2+ \biggl(2k \sum_{j=0}^{M-1}\lVert \ujp  \rVert^2_\frac14\biggr)^2\le \EE(1+ \lvert \bu_0 \rvert^{16} + \Vert \bu_0 \Vert^4_\frac14).
		\end{split}
		\end{equation}
		The last estimate along with \eqref{STAB-DA-max} completes the proof  of \eqref{STAB-DA-2} and hence the whole proposition.
	\end{proof}
	\section{Error analysis of the numerical scheme \eqref{Algo}: Proof of Theorem \ref{THM-MAIN}}\label{Sec-Proof}
	This section is devoted to the analysis of the error $\bee_j=\bu(t_j)-\uj$ at {the time $t_j$} between the exact solution $\bu$ of \eqref{spdes} and the approximate solution given by \eqref{Algo}. Since the precise statement of the convergence rate is already given in Theorem \ref{THM-MAIN}, we proceed directly to the promised proof of Theorem \ref{THM-MAIN}.

	Before giving the proof of Theorem \ref{THM-MAIN} we state and prove the following important result.
	\begin{lem}\label{LEM-Essence}
		Let $\beta$ be as in Theorem \ref{THM-MAIN}. Then,
		\begin{enumerate}[label=(\roman{*})]
			\item \label{Ess-i} 	there exists a constant $C_7>0$ such that 
			\begin{equation}
			\mathbb{E}{\Vert \bu(t)-\bu(s)\Vert^2_{\beta}} \le C_7 [(t-s)^{2-2\beta}+(t-s)^{2(\frac14-\beta)}+(t-s)],
			\end{equation}
			for any $t,s \ge 0$ and $t\neq s$.
			\item \label{Ess-ii} There also exists a positive constant  $C_8$ such that 
			\begin{equation}
			\EE \int_s^t \Vert \bu(t)-\bu(r)\Vert^2_{\frac12+\beta} dr \le  C_8 \left( (t-s)^{\frac32-2\beta}+(t-s)^{2(\frac14-\beta)}+(t-s)^{2-2\beta}\right),
			\end{equation}
			for any $t>s\ge 0$.
		\end{enumerate}
	\end{lem}
	\begin{proof}[Proof of Lemma \ref{LEM-Essence}]
		As in the statement of the lemma we divide the proof into two parts.
		
		\textit{Proof of item \ref{Ess-i}.} Let $t,s\in [0,T]$ such that $t\neq s$. Without loss of generality we assume that $t>s$. Thanks to \eqref{Eq-REM} of Remark \ref{REM-Mild} we have 
		\begin{equation*}
		\begin{split}
		\Vert \bu(t)-\bu(s)\Vert^2_{\beta}\le C \vert \rA^{\beta-\frac14} (\mathrm{I}- e^{-(t-s)\rA})\rA^\frac14 \bu(s)\vert^2+C\biggl \vert \int_s^t A^\beta e^{-(t-r)\rA} B(\bu(r),\bu(r)) dr\biggr \vert^2
		\\
		+ C \biggl\vert\int_s^t \rA^\beta e^{-(t-r)\rA} G(\bu(r))dW(r)\biggr\vert^2.
		\end{split}
		\end{equation*}
		Before proceeding further we recall that there exists a constant $C>0$ such that for any $\gamma>0$ and $t\ge 0$, we have
		
		\begin{equation*} 
		\Vert A^{-\gamma }(\mathrm{I}-e^{-tA}) \Vert_{\mathscr{L}(\h)} \le C t^\gamma. 
		\end{equation*}
		Applying this inequality, the H\"older inequality, Assumption \ref{Assum-B1p}, the It\^o isometry and Assumption \ref{nonlin-g} imply
		\begin{equation*}
		\begin{split}
		\EE	\left(\Vert \bu(t)-\bu(s)\Vert^2_{\beta}\right)\le & C (t-s) \EE \biggl(\int_s^t (t-r)^{-2\beta} \Vert \bu(r)\Vert^2_{\frac14} 
		\Vert \bu(r)\Vert^2_{\frac12-\frac14}  dr\biggr)
		\\
		& +C (t-s)^{2(\frac14-\beta)}\EE \Vert \bu(s)\Vert^2_{\frac14} + \EE \int_s^t  \vert e^{-(t-r)\rA} \rA^\beta G(\bu(r))\vert^2dr
		\\
		\le &  C (t-s)^{2-2\beta}\EE \left(\sup_{r\in [s,t]} \Vert \bu(r)\Vert^2_{\frac14} 
		\sup_{r\in [s,t]}\Vert \bu(r)\Vert^2_{\frac12-\frac14} \right)\\
		& +C [(t-s)^{2(\frac14-\beta)}+(t-s) ]\EE \left(\sup_{r\in [s,t]}\Vert \bu(r)\Vert^4_{\frac14}\right),
		\end{split}
		\end{equation*}
		from which along with \eqref{Eq-10} we easily infer that
		\begin{equation*}
		\EE	\left(\Vert \bu(t)-\bu(s)\Vert^2_{\beta}\right)\le C [(t-s)^{2-2\beta}+(t-s)^{2(\frac14-\beta)}+(t-s)].
		\end{equation*} 
		Thus, we have just finished the proof of the first part of the lemma.
		
		\textit{Proof of item \ref{Ess-ii}.} 
		Let $t>s\ge 0$. Using  \eqref{Eq-REM} of Remark \ref{REM-Mild}, it is not difficult to see that 
		\begin{equation*}
		\begin{split}
		\int_s^t \Vert \bu(t)-\bu(r)\Vert_{\frac12+\beta}^2 dr \le &
		C \int_s^t \biggr(\int_r^t \vert \rA^{\frac12 +\beta } e^{-(t-\tau)\rA}\mb(\bu(\tau),\bu(\tau)) \vert d\tau \biggr)^{2} dr\\
		&\quad + C \int_s^t \biggl\vert \int_r^t\rA^{\frac14+\beta} e^{-(t-\tau)\rA}[\rA^\frac14 G(\bu(\tau))]dW(\tau)\biggr\vert^2 dr \\
		&\quad \quad 	+C \int_s^t \vert \rA^{\beta -\frac14} (e^{-(t-r)A} -\mathrm{I})\rA^{\frac34} \bu(s)\vert^2 dr,
		\end{split}
		\end{equation*}
		from which and the assumption on $\mb$ we infer that 
		\begin{equation*}
		\begin{split}
		\int_s^t \Vert \bu(t)-\bu(r)\Vert_{\frac12+\beta}^2 dr \le & C
		\sup_{0\le \tau \le T}\left(\Vert \bu(\tau)\Vert^2_{\frac14}\Vert \bu(\tau) \Vert^2_{\frac12-\frac14 }\right) \int_s^t \biggr(\int_r^t (t-\tau)^{-\frac12 -\beta }  d\tau \biggr)^{2} dr\\
		&\quad + C \int_s^t \biggl\vert \int_r^t\rA^{\frac14+\beta} e^{-(t-\tau)\rA}[\rA^\frac14 G(\bu(\tau))]dW(\tau)\biggr\vert^2 dr \\
		&\quad \quad 	+C \int_s^t (t-r)^{2(\frac14-\beta)} \Vert \bu(s)\Vert^2_{\frac34} dr.
		\end{split}
		\end{equation*}
		Taking the mathematical expectation and using \eqref{Eq-10} yield
		\begin{equation*}
		\begin{split}
		\EE\left( \mathds{1}_{\Omega_k}\int_s^t \Vert \bu(t)-\bu(r)\Vert_{\frac12+\beta}^2 dr\right) \le & C  (t-s)^{2-2\beta}+ C (t-s)^{2(\frac14-\beta)} \EE\int_0^T \Vert \bu(r)\Vert^2_{\frac12+\beta} dr
		\\
		&\quad + \int_s^t \EE\biggl(\biggl\vert \int_r^t\rA^{\frac14+\beta} e^{-(t-\tau)\rA}\rA^\frac14 G(\bu(\tau))dW(\tau)\biggr\vert^2 \biggl)dr. 
		\end{split}
		\end{equation*}
		Owing to the It\^o isometry, the assumption \ref{nonlin-g} and \eqref{Eq-10}, we obtain 
		\begin{equation*}
		\begin{split}
		\EE\left( \int_s^t \Vert \bu(t)-\bu(r)\Vert_{\frac12+\beta}^2 dr\right) \le \EE\left(\sup_{0\le \tau \le T}(1+ \Vert \bu(\tau)\Vert^2_\frac14)\right) \int_s^t  \int_r^t (t-\tau)^{-\frac12-2\beta} d\tau dr\\
		+ (t-s)^{2-2\beta}+ (t-s)^{2(\frac14-\beta)},
		\end{split}
		\end{equation*}		
		{from which altogether with \eqref{Eq-10} we infer that  there exists a constant $C>0$ such that}
		\begin{equation*}
		\begin{split}
		\EE\left( \int_s^t \Vert \bu(t)-\bu(r)\Vert_{\frac12+\beta}^2 dr\right) \le 
		C  (t-s)^{2-2\beta}+ C (t-s)^{2(\frac14-\beta)}+ C (t-s)^{\frac32-2\beta},
		\end{split}
		\end{equation*}	
		for any $t>s\ge 0$.
	\end{proof}

	We now give the promised proof of Theorem \ref{THM-MAIN}.
	\begin{proof}[Proof of Theorem \ref{THM-MAIN}]
		Since the embedding $\ve_\beta\subset \h$ is continuous for any $\beta \in (0,\frac14)$, it is sufficient to prove the main theorem for $\beta \in (0,\frac14)$.
		
		\noindent Note that the numerical scheme \eqref{Algo} is equivalent to 
		\begin{equation}\label{Int-algo}
		\begin{split}
		(\ujp, w)+\int_{t_j}^{t_{j+1}} \langle \rA \ujp + \pi_N \mb (\uj, \ujp), w\rangle ds = (\uj, w)+
		\int_{t_j}^{t_{j+1}} \langle w, \pi_N G(\uj) dW(s)\rangle
		\end{split}
		\end{equation}
		for any $j \in \{1,\ldots,M\}$ and $w\in \ve$. Integrating \eqref{spdes} and subtracting the resulting equation and the identity \eqref{Int-algo} term by term yield
		\begin{equation}\label{Eq-Error}
		\begin{split}
		(\erjp-\erj, w)+\int_{t_j}^{t_{j+1}} \langle \rA \erjp + \rA(\bu(s)-\bu(t_{j+1}))+ \mb(\bu(s),\bu(s))-\pi_N \mb(\uj, \ujp), w\rangle ds \\= \int_{t_j}^{t_{j+1}} \langle w, [G(\bu(s))-\pi_N G(\uj)]dW(s)\rangle  .
		\end{split}
		\end{equation}
		Observe that if $\bv\in D(\rA^{\frac12+\alpha})$ with $\alpha>\beta $, then $\rA^{2\beta}\bv\in D(\rA^{\frac12 +\alpha-\beta})\subset D(\rA^{\frac12 -\alpha}) $, ${\rA\bv}\in D(\rA^{\alpha-\frac12})$ and the duality product $\langle \rA \bv, \rA^{2\beta}\bv \rangle $ is meaningful. Thus, we are permitted to take $w=2\rA^{2\beta}\erjp$ in \eqref{Eq-Error} and derive that
		\begin{equation*}
		\begin{split}
		\Vert \erjp\Vert^2_{\beta}-\Vert \erj\Vert^2_{\beta}+\Vert \erjp-\erj\Vert^2_{\beta}+2 k \Vert \erjp \Vert^2_{\frac12+\beta}- 2\int_{t_j}^{t_{j+1}} \Vert \rA^{\frac12+\beta}(\bu(s)-\bu(t_{j+1}))\Vert_{\frac12 +\beta }\Vert \erjp \Vert_{\frac12 +\beta} ds\\
		\le 2 \int_{t_j}^{t_{j+1}}\left \vert ( \rA^{\beta-\frac12} [\mb(\bu(s),\bu(s))-\pi_N \mb(\uj, \ujp)], \rA^{\frac12 +\beta }\erjp) \right\vert ds\\  +2\int_{t_j}^{t_{j+1}} \langle \rA^{2\beta}\erjp, [G(\bu(s))-\pi_N G(\uj)]dW(s) \rangle,
		\end{split}
		\end{equation*}
		where we have used the identity $( \bv-\bx, 2\rA^{2\beta}\bv)= \Vert \bv \Vert^2_\beta -\Vert \bx \rVert^2_\beta + \Vert \bv-\bx\Vert^2_\beta $.
		Now, by using the identity $\bv=(\pi_N +[\mathrm{I}-\pi_N] )\bv$, the fact that 
		\begin{equation*}
			\begin{split}
			\mb(\bu(s),\bu(s))-\pi_N \mb(\uj, \ujp)= &\mb(\bu(s),\bu(s))-\pi_N \mb(\bu(t_j),\bu(t_{j+1}))
			\\
			&+\pi_N \mb(\bu(t_j),\bu(t_{j+1}))-  \mb(\uj, \ujp),
			\end{split}
		\end{equation*}
		the Cauchy-Schwarz inequality, the Cauchy inequality $ab\le \frac{a^2}4+b^2$, $a,b>0$ and Assumption \ref{NB-a} we obtain  
		\begin{equation}\label{Eq-Err-1}
		\begin{split}
		\Vert \erjp\Vert^2_{\beta}-\Vert \erj\Vert^2_{\beta}+\Vert \erjp-\erj\Vert^2_{\beta}+k \Vert \erjp \Vert^2_{\frac12+\beta}\le  
		2 \mathcal{L}_j+16 C_0^2 \sum_{i=1}^5 \mathscr{N}_{j,i} +2 \mathscr{W}_j,
		\end{split}
		\end{equation}
		where for each $j\in \{0,\ldots,M-1\}$ the symbols $\mathcal{L}_j$, $\mathscr{N}_{j,i}$, $i=1,\ldots, 5$, and $\mathscr{W}_j$ are defined by
		\begin{align*}
		& \mathcal{L}_j:=\int_{t_j}^{t_{j+1}} \Vert \bu(s)-\bu(t_{j+1})\Vert^2_{\frac12 +\beta } ds, \\
		& \mathscr{N}_{j,1}:=\int_{t_j}^{t_{j+1}} \Vert \bu(s)-\bu(t_{j+1})\Vert^{2}_{\beta} (\Vert \uj\Vert^2_\beta+ \Vert \bu(s)\Vert^2_\beta ) ds, \\ 
		& \mathscr{N}_{j,2} :=\int_{t_j}^{t_{j+1}} \Vert \erjp\Vert^{2}_{\beta} (\Vert \ujp\Vert^2_\beta + \Vert \bu(s)\Vert^2_\beta) ds,\\ &
		\mathscr{N}_{j,3}:=\int_{t_j}^{t_{j+1}} \Vert \bu(s)-\bu(t_j)\Vert^{2}_{\beta} (\vert \ujp\vert^2+ \vert \bu(s)\vert^2) ds,\\
		& \mathscr{N}_{j,4} :=\int_{t_j}^{t_{j+1}} \Vert \erj\Vert^{2}_{\beta} (\vert \ujp\vert^2+ \vert \bu(s)\vert^2) ds,\\ &  \mathscr{N}_{j,5}:= \int_{t_j}^{t_{j+1}} \Vert (\mathrm{I}-\pi_N) \mb(\bu(s),\bu(s))\Vert^2_{\beta -\frac12} ds,\\
		& {\mathscr{W}_j:= \int_{t_j}^{t_{j+1}}\langle \rA^{2\beta}\erjp , [G(\bu(s))- \pi_NG(\uj)]dW(s) \rangle.} 
		\end{align*}
		
		Let $m\in [1,M]$ an arbitrary integer. Summing \eqref{Eq-Err-1} from $j=0$ to $m-1$ , multiplying by $\mathds{1}_{\Omega_k}$, taking the mathematical expectation, and finally  taking the maximum over $m\in [1,M]$ imply
		\begin{equation*}
		\begin{split}
		{\max_{1\le m\le M}} \EE\left[\mathds{1}_{\Omega_k} \Vert \mathbf{e}^m\Vert^2_{\beta}\right]+\sum_{j=0}^{M-1} \EE\left[\mathds{1}_{\Omega_k}\Vert \erjp-\erj\Vert^2_{\beta}\right]+k \sum_{j=1}^M\EE\left[\mathds{1}_{\Omega_k}\Vert \erj \Vert^2_{\frac12+\beta}\right]\\\le  
		\EE \Vert \mathbf{e}^0\Vert^2_{\beta}+16 C^2_0 \sum_{j=0}^{M-1}\sum_{i=1}^5  \EE\left[\mathds{1}_{\Omega_k} \mathscr{N}_{j,i}\right]+2\sum_{j=0}^{M-1} \EE\left[\mathds{1}_{\Omega_k}\mathcal{L}_j\right] +2 \max_{1\le m \le M} \sum_{j=0}^{m-1}\EE\left[\mathds{1}_{\Omega_k}\mathscr{W}_j\right].
		\end{split}
		\end{equation*}
		Invoking the two items of Lemma \ref{LEM-Essence} and the fact that $\Vert \bu(s)\Vert^2_{\beta} + \max_{0\le j\le M}\Vert \uj \Vert^2_{\beta}\le f(k)$ on the set $\Omega_k$ we infer that 
		\begin{equation}\label{Eq-Err-2}
		\begin{split}
		\max_{1\le m\le M} \EE\left[\mathds{1}_{\Omega_k}\Vert \mathbf{e}^m\Vert^2_{\beta}\right]+\sum_{j=0}^{M-1} \EE\left[\mathds{1}_{\Omega_k}\Vert \erjp-\erj\Vert^2_{\beta}\right]+k \sum_{j=1}^M\EE\left[\mathds{1}_{\Omega_k}\Vert \erj \Vert^2_{\frac12+\beta}\right]\\ \le \EE \Vert \mathbf{e}^0\Vert^2_{\beta}+16 C_0^2 kf(k)\sum_{j=0}^{M-1}\EE\left(\mathds{1}_{\Omega_k}[\Vert \mathbf{e}^{j+1}\Vert^2_{\beta}+\Vert \mathbf{e}^j\Vert^2_{\beta}]\right)+ 2 C_8 f(k)M k [\Psi(k)+k^{1+\frac12-\beta}]\\+ 64 C_0^2 C_8 [f(k)]^2 M k [\Psi(k)+k]+16C_0^2 \sum_{j=0}^{M-1}\mathscr{N}_{j,5}+2 \max_{1\le m\le M}\sum_{j=0}^{m-1}\EE\left[\mathds{1}_{\Omega_k}\mathscr{W}_j\right],
		\end{split}
		\end{equation}
		where $\psi(k):=k^{2-2\beta}+ k^{2(\frac14-\beta)}$. Now, thanks to Assumption \ref{Assum-B1p} we have
		\begin{align*}
		\mathds{1}_{\Omega_k} \int_{t_j}^{t_{j+1}} \Vert (\textrm{I}-\pi_N) \mb(\bu(s), \bu(s))\Vert^2_{\beta -\frac12} ds&= \mathds{1}_{\Omega_k}  \int_{t_j}^{t_{j+1}} \sum_{n=N+1}^{\infty} \lambda_n^{2\beta-1} \lvert \mb_n(\bu(s), \bu(s))\rvert^2 ds
		\\
		&\le \lambda_{N}^{2\beta-1} \int_{t_j}^{t_{j+1}} \mathds{1}_{\Omega_k}  \sum_{n=0}^\infty \vert \mb_n(\bu(s), \bu(s)) \rvert^2 ds
		\\
		&\le \lambda_N^{2\beta-1} \int_{t_j}^{t_{j+1}} \mathds{1}_{\Omega_k}  \vert \mb(\bu(s), \bu(s) ) \vert^2 ds
		\\
		&\le C \lambda_N^{2\beta-1} k \sup_{s\in [0,T]} \Vert \bu(s)\Vert^4_\frac14.
		\end{align*}
		Hence, owing to \eqref{Eq-10} we find a constant $C>0$ such that
		$$  \EE \mathds{1}_{\Omega_k} \int_{t_j}^{t_{j+1}} \Vert (\textrm{I}-\pi_N) \mb(\bu(s), \bu(s))\Vert^2_{\beta -\frac12} ds\le C \lambda_N^{2\beta-1}k.$$
		Notice also that
		\begin{align*}
		& \sum_{j=0}^{M-1}\Vert \erjp\Vert^2_{\beta}(\Vert \ujp\Vert^2_\beta + \Vert \bu(s)\Vert^2_\beta) \\
		&\quad \quad \quad = \sum_{j=0}^{M-1} \Vert \ujp -\uj+\uj-\bu(t_j)+\bu(t_j) -\bu(t_{j+1})\Vert^2_{\beta} (\Vert \ujp\Vert^2_\beta + \Vert \bu(s)\Vert^2_\beta)
		\\
		&\quad \quad \quad  \le 3 \sum_{j=0}^{M-1}\left(\Vert \ujp-\uj \Vert^2_\beta+ \Vert \mathbf{e}^j \Vert^2_\beta + \Vert \bu(t_j)-\bu(t_{j+1})\Vert^2_\beta\right)(\max_{0\le j\le M}\Vert \ujp\Vert^2_\beta + \Vert \bu(s)\Vert^2_\beta).
		\end{align*}
		Therefore,
		\begin{equation*}
		\begin{split}
		\EE\left(\mathds{1}_{\Omega_k} \sum_{j=0}^{M-1}\Vert \erjp\Vert^2_{\beta}(\Vert \ujp\Vert^2_\beta + \Vert \bu(s)\Vert^2_\beta)\right)  - Cf(k) \EE \sum_{m=0}^{M-1}\Vert \mathbf{e}^j \Vert^2_\beta + f(k)C_7 [\psi(k)+k]\\
		\le  C \left(\EE \left( \sum_{j=0}^{M-1} \Vert \ujp-\uj \Vert^2_\beta\right)^2\right)^\frac12 \left(\EE\max_{0\le j\le M} \Vert \uj \Vert^4_\beta+ \EE \sup_{s\in [0,T]} \Vert \bu(s)\Vert^4_\beta\right)^\frac12.
		\end{split}
		\end{equation*}
		As long as the initial data is concerned, we have 
		\begin{align}
		\EE \Vert \mathbf{e}^0\Vert^2_\beta&= \Vert [\pi_N+(\mathrm{I}-\pi_N)]\bu_0-\pi_N \bu_0\Vert^2_\beta
		\\
		&\le \sum_{n=N+1}^\infty \lambda^{2(\beta-\frac14)}_n \lambda_N^{\frac12} \vert \bu_{0,n}\vert^2
		\\
		&\le \lambda^{2(\beta-\frac14)}_N \Vert \bu_0\Vert^2.
		\end{align}
		From all the above observations, \eqref{Eq-Err-2}, Assumption \ref{Assum-B1p},  \eqref{EQ:Stabl1}-\eqref{STAB-l2-B} and \eqref{STAB-DA-2} we infer that there exists a constant $C_9>0$ such that
		\begin{equation}\label{Eq-Err-2-B}
		\begin{split}
		\max_{1\le m\le M}\EE\left[\mathds{1}_{\Omega_k}\Vert \mathbf{e}^m\Vert^2_{\beta}\right]+\sum_{j=0}^{M-1} \EE\left[\mathds{1}_{\Omega_k}\Vert \erjp-\erj\Vert^2_{\beta}\right]+k \sum_{j=1}^M\EE\left[\mathds{1}_{\Omega_k}\Vert \erj \Vert^2_{\frac12+\beta}\right]
		\\ 
		\le C_9 f(k) [\Psi(k)+k^{1+\frac12-\beta}]+ C_9 f(k) [\Psi(k)+k]+C_9\left( \lambda_N^{2\beta-1} +\lambda^{2(\beta-\frac14)}\right)+2 \max_{1\le m \le M }\sum_{j=0}^{m-1}\EE\left[\mathds{1}_{\Omega_k}\mathscr{W}_j\right]
		\\
		+ C_9 kf(k) \sum_{m=0}^{M-1} \max_{1\le j\le m}\EE \left[\mathds{1}_{\Omega_k} \Vert \erj\Vert^\beta \right]+ 16 C^2_0 kf(k) \max_{1\le m \le M} \EE\left[\mathds{1}_{\Omega_k} \Vert\mathbf{e}^m\Vert^2_{\beta} \right].
		\end{split}
		\end{equation}
		
		Now we deal with the term containing $\mathscr{W}_j$. After subtracting from $\mathscr{W}_j$ the martingale $M_0$ with mean zero defined by
		\begin{align*}
		M_0={ \int_{t_j}^{t_{j+1}}\langle \rA^{\beta}\erjp, \rA^{\beta} [G(\bu(s))- \pi_NG(\uj)]dW(s) \rangle,}
		\end{align*}
		then taking the mathematical expectation, using the Young inequality and the It\^o isometry give 
		\begin{align*}
		\mathbb{E}\mathds{1}_{\Omega_k}\mathscr{W}_{j} &\leq C\mathbb{E}\mathds{1}_{\Omega_k}\bigg\lVert\int_{t_j}^{t_{j+1}} [G(\bu(s))- \pi_NG(\uj)]dW(s)\bigg\rVert^2_\beta+ \frac{1}{4}\mathbb{E}\mathds{1}_{\Omega_k}\lVert \erjp-\erj \rVert^2_\beta
		\\
		&\leq C\int_{t_j}^{t_{j+1}} \mathbb{E}\mathds{1}_{\Omega_k}\lVert G(\bu(s))- \pi_NG(\uj)\rVert^2_{\cL(\mathscr{H}, V_\beta) }ds+ \frac{1}{4}\mathbb{E}\mathds{1}_{\Omega_k}\lVert \erjp-\erj \rVert^2_\beta
		\\
		&\leq \sum_{i=1}^3\EE[\mathds{1}_{\Omega_k}\mathscr{W}_{j,i}]+\frac{1}{4}\mathbb{E}\mathds{1}_{\Omega_k}\lVert \erjp-\erj \rVert^2_\beta,
		\end{align*}
		where the first two symbols $\mathscr{W}_{j,i}$, $i\in\{1,2\}$ satisfy the following equalities and inequalities  
		\begin{align}
		\EE[\mathds{1}_{\Omega_k}\mathscr{W}_{j,1}] 	&= C\int_{t_j}^{t_{j+1}} \mathbb{E}\mathds{1}_{\Omega_k}\lVert \pi_N G(\bu(s))- \pi_NG(\bu(t_{j}))\rVert^2_{\cL(\mathscr{H}, V_\beta)} ds
		\nonumber\\
		&\leq CC_3^2 \int_{t_j}^{t_{j+1}} \mathbb{E} \lVert  \bu(s)- \bu(t_{j})\rVert_{\beta}^2 ds
		\nonumber \\
		&\leq CC_3^2C_7^2 k[k^{2-2\beta}+k^{2(\frac14-\beta)}+k];\nonumber
		\end{align}
		\begin{align}
		\EE[\mathds{1}_{\Omega_k}\mathscr{W}_{j,2}] 	&= C\int_{t_j}^{t_{j+1}} \mathbb{E}\mathds{1}_{\Omega_k}\lVert \pi_NG(\bu(t_{j}))-\pi_NG(U^j)\rVert^2_{\cL(\mathscr{H}, V_\beta) }ds
		\nonumber\\
		&\leq CC_3^2 k \mathbb{E}\mathds{1}_{\Omega_k} \lVert  \mathbf{e}^j \rVert_{\beta}^2,\nonumber 
		\end{align}
		where Lemma \ref{LEM-Essence}  was used to get the last line.
		
		\noindent The third term $\mathscr{W}_{j,3}$ satisfies 
		\begin{align*}
		\EE[\mathds{1}_{\Omega_k}\mathscr{W}_{j,3}]=&\int_{t_j}^{t_{j+1}} \mathbb{E}\left(\mathds{1}_{\Omega_k}\lVert (\mathrm{I}-\pi_N) G(\bu(s)) \rVert^2_{\cL(\mathscr{H}, V_\beta)}\right)ds
		\\
		=& \int_{t_j}^{t_{j+1}} \mathbb{E}\left(\mathds{1}_{\Omega_k} \sum_{n=N+1}^{\infty} \lambda_n^{2(\beta-\frac14)} \lambda_n^{\frac12} \sup_{h \in \mathscr{H}, \Vert h\Vert_{\mathscr{H}\le 1}}\lvert G_n(\bu(s))h\rvert^2\right) ds
		\\
		\le&  \lambda_{N}^{2(\beta-\frac14)} \int_{t_j}^{t_{j+1}} \mathbb{E}\left(\mathds{1}_{\Omega_k} \sum_{n=1}^{\infty}  \lambda_n^{\frac12} \sup_{h \in \mathscr{H}, \Vert h\Vert_{\mathscr{H}\le 1}}\lvert G_n(\bu(s))h\rvert^2\right) ds
		\\
		\le & \lambda_{N}^{2(\beta-\frac14)} k \mathbb{E}\left(\mathds{1}_{\Omega_k} \sup_{s\in[0,T]} \Vert G(\bu(s))\Vert^2_{\cL(\mathscr{H},\ve_\frac14)}\right).
		\end{align*}
		Now, using Assumption \ref{nonlin-g} and the estimate \eqref{Eq-10} we infer that 
		\begin{equation}
		\EE[\mathds{1}_{\Omega_k}\mathscr{W}_{j,3}] \le C C_3^2 \lambda_{N}^{2(\beta-\frac14)} k,\nonumber 
		\end{equation}
		for any $j\in [0,M]$.
		Thus, summing up we have obtained that
		\begin{equation*}
		\begin{split}
		2 \max_{1\le m \le M} \sum_{j=0}^{m-1}\EE\left[\mathds{1}_{\Omega_k}\mathscr{W}_j\right]\le C C_3^2 C_7^2 T [\psi(k)+k]+ C C_3^2 T \lambda_{N}^{2(\beta-\frac14)}\\ + C C_3^2 k\sum_{m=0}^{M-1} \max_{1\le j\le m }\EE[\mathds{1}_{\Omega_k}\Vert \mathbf{e}^j\Vert^2_\beta ]+ \frac12 \sum_{m=0}^{M-1}\EE\left(\mathds{1}_{\Omega_k} \Vert \mathbf{e}^{m+1}-\mathbf{e}^m \Vert^2_\beta \right).
		\end{split}
		\end{equation*}
		By plugging this last estimate into \eqref{Eq-Err-2}, we  find a constant $C_{10}>0$ such that 
		\begin{equation*}
		\begin{split}
		\max_{1\le m\le M} \EE\left[\mathds{1}_{\Omega_k}\Vert \mathbf{e}^m\Vert^2_{\beta}\right]+ \sum_{j=0}^{M-1} \EE\left[\mathds{1}_{\Omega_k}\Vert \erjp-\erj\Vert^2_{\beta}\right]+2k \sum_{j=1}^M\EE\left[\mathds{1}_{\Omega_k}\Vert \erj \Vert^2_{\frac12+\beta}\right]
		\\ 
		\le C_{10} f(k) [\Psi(k)+k+k^{1+\frac12-\beta}]+ C_{10} f(k) [\Psi(k)+k]+C_{10}\lambda_N^{2\beta-1}+C_{10}\lambda_N^{2(\beta-\frac14)}
		\\
		+ C_{10} k[f(k)+1] \sum_{m=0}^{M-1}\max_{1\le j\le m}\EE \left[\mathds{1}_{\Omega_k}  \Vert \erj\Vert^\beta \right].
		\end{split}
		\end{equation*}
		Now, an application of the discrete Gronwall lemma yields
		\begin{equation*}
		\begin{split}
		\max_{1\le m\le M} \EE\left[\mathds{1}_{\Omega_k}\Vert \mathbf{e}^m\Vert^2_{\beta}\right]+ \sum_{j=0}^{M-1} \EE\left[\mathds{1}_{\Omega_k}\Vert \erjp-\erj\Vert^2_{\beta}\right]+2k \sum_{j=1}^M\EE\left[\mathds{1}_{\Omega_k}\Vert \erj \Vert^2_{\frac12+\beta}\right]\\ 
		\le \left( C_{10} f(k) [\Psi(k)+k+k^{1+\frac12-\beta}]+ C_{10} f(k) [\Psi(k)+k]+C_{10} \lambda_N^{2\beta-1}+C_{10}\lambda_N^{2(\beta-\frac14)}\right)e^{C_{10}T[f(k)+1]}.
		\end{split}
		\end{equation*}
		Since 
		$$ \min\{k^{2-2\beta}, k^{1+\frac12-\beta}, k^{2(\frac14-\beta)}, k\}=k^{2(\frac14-\beta)} \text{ and }  \min\{\lambda_N^{2(\beta-\frac14)}, \lambda_N^{2\beta-1} \}=\lambda_N^{2(\beta-\frac14)},
		$$ 
		for any $\beta \in [0,\frac14)$, and  $k^\eps f(k)=k^{\eps}\log{k^{-\eps}}\le \frac 12$, then for any $k>0$ and $\eps\in \Big(0, 2(\frac14-\beta)\Big)$,  we derive that there exists a constant $C>0$ such that
		\begin{equation}
		\begin{split}
		\max_{1\le m\le M}\EE\left[\mathds{1}_{\Omega_k}\Vert \mathbf{e}^m\Vert^2_{\beta}\right]+ \sum_{j=0}^{M-1} \EE\left[\mathds{1}_{\Omega_k}\Vert \erjp-\erj\Vert^2_{\beta}\right]\\+2k \sum_{j=1}^M\EE\left[\mathds{1}_{\Omega_k}\Vert \erj \Vert^2_{\frac12+\beta}\right]\le C k^{-2\eps} [k^{2(\frac14-\beta)}+\lambda_N^{-2(\frac14-\beta)}].
		\end{split}
		\end{equation}
		This estimate completes the proof of the Theorem \ref{THM-MAIN}.
	\end{proof}
	
	\section{Motivating Examples}\label{motivations}
	In this section we give two examples of evolution equations to which we can apply our abstract result.
	\subsection{Stochastic GOY and Sabra shell models}
	The first examples we can take is the GOY and Sabra shell models. To describe this model let us denote by $\mathbb{C}$ the field of complex numbers, $\mathbb{C}^{\mathbb{N}}$ the set of all $\mathbb{C}$-valued sequences, and we set 
	\begin{equation*}
	\h=\left\{\bu=(\bu_n)_{n\in \mathbb{N}} \subset \mathbb{C};  \sum_{n=1}^{\infty}\lvert \bu_n \rvert^2 <\infty \right\}.
	\end{equation*}
	Let $k_0$ be a positive number and $\lambda_n=k_0 2^n$ be a sequence of positive numbers.  The space $\h$ is a separable Hilbert space when endowed with the scalar product defined by 
	$$ \langle \bu, \bv \rangle =\sum_{k=1}^\infty \bu_k \bar{\bv}_k, \; \text{ for } \bu,\; \bv \in \h,$$
	where $\bar{z}$ denotes the  conjugate of any complex number $z$.
	
	We define a linear map $\rA$ with domain $$ D(\rA)= \{ \bu \in \h ; \sum_{n=1}^\infty \lambda_n^4 \lvert \bu_n\rvert^2 <\infty\},$$
	by setting $$ \rA \bu=(\lambda_n^2 \bu_n )_{n\in \mathbb{N}}, \text{ for } \bu \in D(\rA).$$
	It is not hard to check that $\rA$ is a self-adjoint and strictly positive operator. Moreover, the embedding $D(\rA^\alpha)\subset D(\rA^{\alpha+\eps})$ is compact for any $\alpha \in \mathbb{R}$ and $\eps>0$. Thanks to this observation we can and will assume that there exists an orthonormal basis  $\{ \psi_n; n \in \mathbb{N}\} $ of $\h$ such that $$ \rA \psi_n = \lambda_n \psi_n .$$  
	We can  characterize the spaces $D(\rA^\alpha )$, $\alpha \in \mathbb{R}$ as follow
	\begin{equation*}
	D(\rA^\alpha)= \{\bu=(\bu_n)_{n \in \mathbb{N}}\subset \mathbb{C}; \sum_{n=1}^\infty \lambda_n^{4\alpha} \lvert \bu_n\rvert^2 <\infty \}.
	\end{equation*}
	For any $\alpha \in \err$ the space $\ve_\alpha=D(\rA^\alpha)$ is a separable Hilbert space when equipped with the scalar product
	\begin{equation}
	(( \bu, \bv ))_\alpha=\sum_{k=1}^{\infty} \lambda_k^{4\alpha} \bu_k \bar{\bv}_k, \text{ for } \bu,\; \bv\in \ve_\alpha. 
	\end{equation}
	The norm associated to this scalar product will be denoted by $\Vert \bu \Vert_\alpha$, $\bu \in \ve_\alpha$.
	In what follows we set $\ve=D(\rA^\frac12)$.
	
	Now,  let $\alpha_0> \frac12$ and $\{\mathfrak{w}_j; \; j \in \mathbb{N} \}$ be a sequence of mutually independent and identically distributed standard Brownian motions  on  filtered complete probability space $\mathfrak{U}=(\Omega, \mathscr{F}, \mathbb{F}, \mathbb{P})$ satisfying the usual condition. We  set 
	
	$$ W(t)=\sum_{n=0}^\infty \lambda_n^{-\alpha_0} \mathfrak{w}_n(t)\psi_n . $$
	The process $W$ defines a $\h$-valued {process} with covariance $\rA^{-2\alpha_0}$ which is of trace class. We also consider a Lipschitz map $g: [0,\infty) \to \mathbb{R} $ such that $\lvert g(0)\rvert<\infty$. We define a map $G:\h \to \mathscr{L}(\h, \ve_{\frac14})$ defined by 
	$$ G(u)h=g(\lVert u \rVert_0 ) h, \text{ for any } u\in \h, h \in \h.$$ 
	This map satisfies Assumption \ref{nonlin-g}. 
	
	With the above notation, the stochastic evolution equation describing our randomly perturbed GOY and Sabra shell models is given by 
	{
		\begin{equation}
		\label{spdes-1}
		\left\{
		\begin{split}
		&	{d \bu} = [ \rA \bu + \mb(\bu,\bu)]dt+ G(u) dW,
		\\
		&	\bu(0)=\bu_0,
		\end{split}
		\right.
		\end{equation}}
	where  $\mb(\cdot\,,\,\cdot)$ is a bilinear map defined on $\ve\times \ve$ taking values in the dual space $\ve^\ast$. More precisely, we assume that 
	the nonlinear term 
	\begin{equation*}
	\begin{split}
	\mb:\; &\mathbb{C}^{\mathbb{N}}\times\mathbb{C}^{\mathbb{N}}\to \mathbb{C}^{\mathbb{N}},\\
	&(\bu,\bv)\mapsto \mb(\bu,\bv)=(b_1(\bu,\bv),\ldots, b_n(\bu,\bv),\dots)
	\end{split}
	\end{equation*} 
	for the GOY shell model (see \cite{G}) is defined by
	\begin{align*}
	b_n (\bu,\bv) &:= \left(\mb(\bu,\bv) \right)  _{n} \\
	& :=i\lambda_{n}\left(  \frac{1}{4}\overline{v}_{n-1}%
	\overline{u}_{n+1}-\frac{1}{2}\left(  \overline{u}_{n+1}\overline{v}%
	_{n+2}+\overline{v}_{n+1}\overline{u}_{n+2}\right)
	+\frac{1}{8}\overline {u}_{n-1}\overline{v}_{n-2}\right),
	\end{align*} 
	and for the Sabra shell model, it is defined by 
	\begin{align*}
	b_n (\bu,\bv) := &\left(\mb(\bu,\bv) \right)  _{n}:=
	\frac{i}{3}\lambda_{n+1}\left[  
	\overline{v}_{n+1}u_{n+2}+ 2
	\overline{u}_{n+1}v_{n+2}\right] \\
	&  +\frac{i}{3}\lambda_{n}\left[   \overline{u}%
	_{n-1}v_{n+1}-  \overline{v}_{n-1}u_{n+1}\right] \\
	&  +\frac{i}{3}\lambda_{n-1}\left[    2  u_{n-1}v_{n-2}+ u_{n-2}v_{n-1}\right],
	\end{align*}
	for any $\bu=(u_1,\ldots, u_n, \dots)\in \mathbb{C}^{\mathbb{N}}$ and $\bv=(v_1,\ldots,v_n,\ldots)\in \mathbb{C}^{\mathbb{N}}$. 
	
	\begin{lem}\label{NLB}
		\begin{enumerate}[label=(\alph{*})]
			\item \label{NB-Sa}	For any non-negative numbers $\alpha$ and $\beta$ such that $\alpha+\beta \in (0,\frac12]$, there exists a constant $c_0>0$ such that 
			\begin{equation}
			\Vert \mb(\bu,\bv)\Vert_{-\alpha}\le c_0 \begin{cases}
			\Vert \bu \Vert_{\frac12-(\alpha+\beta)} \Vert \bv \Vert_\beta\quad  \text{ for any } \bu \in \ve_{\frac12-(\alpha+\beta)}, \bv \in \ve_{\beta}\\
			\Vert \bu \Vert_{\beta} \Vert \bv \Vert_{\frac12-(\alpha+\beta)}\quad \text{ for any } \bv \in \ve_{\frac12-(\alpha+\beta)}, \bu \in \ve_{\beta}.
			\end{cases}
			\end{equation} 
			\item \label{NB-Sb} For any $\bu \in \h, \bv \in \ve$
			\begin{equation}
			\langle \bb(\bu,\bv),\bv\rangle=0.
			\end{equation}
		\end{enumerate}
		
	\end{lem}
	\begin{proof}
		The item \ref{NB-Sb} was proved in \cite[Proposition 1]{Constantin}, thus we omit its proof. 
		
		Item \ref{NB-Sa} can be viewed as a generalization of \cite[Proposition 1]{Constantin}. 
		We will just prove the latter item for the Sabra shell model since the proofs for the two models are very similar. 
		Let $\bu \in \ve_{\frac12-(\alpha+\beta)}$, $\bv \in \ve_{\beta}$, and $\bw \in \ve_{\alpha}$ such that $\Vert \bw \Vert_\alpha \le 1$. We have 
		\begin{align*}
		\vert \langle \mb(\bu, \bv), \bw \rangle \vert =&\lvert \sum_{n=1}^\infty b_n(\bu,\bv) \bar{\bw}_n \rvert\le \sum_{n=1}^\infty \vert b_n(\bu,\bv)\vert \vert \bw_n\vert\\
		\le& \frac13 \sum_{n=1}^\infty \lambda_{n+1} \left(\lvert \bu_{n+1}\rvert\cdot \vert \bv_{n+2}\vert + \vert \bu_{n+2} \vert \cdot \lvert \bv_{n+1}\vert \right)\lvert \bw_n \rvert\\
		&+ \frac13 \sum_{n=1}^\infty \lambda_{n} \left(\lvert \bu_{n-1}\rvert\cdot \vert \bv_{n+1}\vert + \vert \bu_{n+1} \vert \cdot \lvert \bv_{n-1}\vert \right)\lvert \bw_n \rvert\\
		&+ \frac13 \sum_{n=1}^\infty \lambda_{n-1} \left(\lvert \bu_{n-1}\rvert\cdot \vert \bv_{n-2}\vert + \vert \bu_{n-2} \vert \cdot \lvert \bv_{n-1}\vert \right)\lvert \bw_n \rvert\\
		\le& I_1 + I_2 +I_3.
		\end{align*}
		For the term $I_1$ we have 
		\begin{align*}
		I_1\le &\frac13 \sum_{n=1}^\infty \lambda_{n+1} \lvert \bu_{n+1}\rvert \cdot \lvert \bv_{n+2}\rvert \lvert \bw_n \rvert + \frac13 \sum_{n=1}^\infty  \lambda_{n+1} \lvert \bu_{n+2}\rvert \cdot \lvert \bv_{n+1}\rvert \lvert \bw_n \rvert \\
		\le &  I_{1,1}+ I_{1,2}.
		\end{align*} 
		We will treat the term $I_{1,1}$. By H\"older's inequality we have 
		\begin{align*}
		I_{1,1}&\le \frac13 \sum_{n=1}^\infty k_0 2 \lambda^{1-2\alpha} \lvert \bu_{n+1}\rvert \cdot \lvert \bv_{n+2} \lvert \lambda_n^{2\alpha} \lvert \bw_n \rvert\\
		&\le \frac23 k_0  \left (\sum_{n=1}^\infty k_0 2 \lambda_n^{2-4(\alpha+\beta)} \lvert \bu_{n+1}\rvert^2  \lambda_n^{4\beta} \lvert \bv_{n+2} \vert^2\right)^\frac12 \left(\sum_{n=1}^\infty \lambda_n^{4\alpha}\lvert \bw_n \rvert^2\right)^\frac12.
		\end{align*}
		Since $\Vert \bw \Vert_\alpha \le 1$ and $\lambda_{n+p}= k_0^p 2^p \lambda_n$  we can find a constant $C>0$ depending only on $\alpha, \beta$ and $k_0$ such that 
		\begin{align*}
		I_{1,1}\le C \left(\max_{k\in \mathbb{N}} \lambda_{n+1}^{2-4(\alpha+\beta)} \lvert \bu_{n+1}\rvert^2 \right)^\frac12 \left(\sum_{n=1}^\infty \lambda_{n+2 }^{4\beta} \lvert \bv_n \rvert^2 \right)^\frac12
		\\
		\le C \left( \sum_{n=1}^\frac12 \lambda_{n+1}^{4[\frac12-(\alpha+\beta)]} \lvert \bu_{n+1}\rvert^2 \right)^\frac12  \left(\sum_{n=1}^\infty \lambda_{n+2 }^{4\beta} \lvert \bv_n \rvert^2 \right)^\frac12,
		\end{align*}
		from which we easily derive that 
		\begin{equation*}
		I_{1,1}\le C \Vert \bu \Vert_{\frac12-(\alpha+\beta)} \Vert \bv \Vert_{\beta}.
		\end{equation*}
		One can use an analogous argument to show that 
		\begin{equation*}
		I_{1,2}\le C \Vert \bu \Vert_{\frac12-(\alpha+\beta)} \Vert \bv \Vert_{\beta}.
		\end{equation*}
		Hence, 
		\begin{equation*}
		I_{1}\le C \Vert \bu \Vert_{\frac12-(\alpha+\beta)} \Vert \bv \Vert_{\beta}.
		\end{equation*}
		Using a similar argument we can also prove that for any non-negative numbers $\alpha$ and $\beta $ satisfying $\alpha+\beta \in (0,\frac12]$ there exists a constant $C>0$ such that
		\begin{equation*}
		I_2+I_3 \le C \Vert \bu \Vert_{\frac12-(\alpha+\beta)} \Vert \bv \Vert_{\beta},
		\end{equation*}
		for any $\bu \in \ve_{\frac12-(\alpha+\beta)}$ and $\bv \in \ve_{\beta}$. 
		Therefore, for any non-negative numbers $\alpha$ and $\beta $ satisfying  $\alpha+\beta \in (0,\frac12]$ we can find a constant $C>0$ such that 
		\begin{equation*}
		\Vert \mb(\bu,\bv)\Vert_{-\alpha}\le C \Vert \bu \Vert_{\frac12-(\alpha+\beta)} \Vert \bv \Vert_{\beta},
		\end{equation*}
		for any $\bu \in \ve_{\frac12-(\alpha+\beta)}$ and $\bv \in \ve_{\beta}$. 
		Interchanging the role of $\bu$ and $\bv$ we obtain that for any two numbers $\alpha$ and $\beta$ as above there exists a positive constant $C$ such that 
		\begin{equation*}
		\Vert \mb(\bu,\bv)\Vert_{-\alpha}\le C \Vert \bv \Vert_{\frac12-(\alpha+\beta)} \Vert \bu \Vert_{\beta},
		\end{equation*}
		for any $\bv \in \ve_{\frac12-(\alpha+\beta)}$ and $\bu \in \ve_{\beta}$. Thus, we have just completed the proof of the lemma for the Sabra shell model. As we mentioned earlier, the case of the GOY model can be dealt with a similar argument.
	\end{proof}
	For more mathematical results related to shell models we refer to \cite{Flandoli1}, \cite{HB+BF-12}, \cite{Flandoli3}, and references therein. 
	\subsection{Stochastic nonlinear heat equation}
	Let $\mo$ be a bounded domain of $\mathbb{R}^d$, $d=1,2$. We assume that its boundary $\partial \mo$ is of class $\mathcal{C}^\infty$.  
	Throughout this section we will denote by $\h^\theta(\mo)$, $\theta\in \mathbb{R}$, the (fractional) Sobolev spaces as defined in \cite{Triebel} and $\h^1_0(\mo)$ be the space of functions $\bu \in \h^1$ such that $\bu_{\vert _{\mo}}=0$. In particular, we set $\h=\el^2(\mo)$ and we denote its scalar product by $(\cdot, \cdot)$.
	
	We define a continuous bilinear map $\mathfrak{a}:\h^1_0(\mo)\times \h^1_0(\mo) \to \mathbb{R}$  by setting 
	\begin{equation*}
	\mathfrak{a}(\bu, \bv)=(\nabla \bu, \nabla \bv),
	\end{equation*}
	for any $\bu, \bv\in \h^1_0(\mo)$. Thanks to the Riesz representation there exists a densely linear map $\rA$ with domain $D(\rA)\subset \h$ such that 
	\begin{equation*}
	\langle \rA \bv, \bu \rangle =\mathfrak{a}(\bv,\bu),
	\end{equation*}
	for any $\bu, \bv\in \h^1_0(\mo)$. It is well known that $\rA$ is a self-adjoint and definite positive and its eigenfunctions $\{ \psi_n; n \in \mathbb{N}\}\subset \mathcal{C}^\infty(\mo)$ form an orthonormal basis of $\h$. The family of eigenvalues associated to $\{ \psi_n; n \in \mathbb{N}\}$ is denoted by $\{\lambda_n; n \in \mathbb{N}\}$. Observe that the asymptotic behaviour of the eigenvalues is given by $\lambda_n \sim \lambda_1 n^{\frac2d}$.
	For any $\alpha\in \mathbb{R}$ we set $\ve_\alpha=D(\rA^\alpha)$, in particular we put $\ve:=D(\rA^\frac12)$. We always understand that the norm in $\ve_\alpha$ is denoted by $\lVert \cdot \rVert_0 $.

	Now,  let $\alpha_0> \frac{d+1}{4}$ and $\{\mathfrak{w}_j; \; j \in \mathbb{N} \}$ be a sequence of mutually independent and identically distributed standard Brownian motions  on  filtered complete probability space $\mathfrak{U}=(\Omega, \mathscr{F}, \mathbb{F}, \mathbb{P})$ satisfying the usual condition. We set 
	
	$$ W(t)=\sum_{n=0}^\infty \lambda_n^{-\alpha_0} \mathfrak{w}_n(t)\psi_n . $$
	The process $W$ defines a $\h$-valued with covariance $\rA^{-2\alpha_0}$ which is of trace class. We also consider a Lipschitz map $g: [0,\infty) \to \mathbb{R} $ such that $\lvert g(0)\rvert<\infty$. We define a map $G:\h \to \mathscr{L}(\h, \ve_{\frac14})$ defined by 
	$$ G(u)h=g(\lVert u \rVert_0 ) h, \text{ for any } u\in \h, h \in \h.$$ 
	This map satisfies Assumption \ref{nonlin-g}.

	The second example we can treat is the stochastic nonlinear heat equation 
	\begin{subequations}\label{RDE}
		\begin{align}
		& d{\bu}-[\Delta \bu -\lvert \bu\vert \bu]dt ={ g(\rVert \bu\lVert_0 )dW},\\
		& \bu=0\text{ on } \partial \mo,\\
		& \bu(0,x)=\bu_0 \quad x\in \mo.
		\end{align}
	\end{subequations}
	This stochastic system can be rewritten as an abstract stochastic evolution equation 
	\begin{equation*}
	d\bu +[\rA\bu  + B(\bu,\bu)] dt=G(\bu)dW, \qquad  \bu(0)=\bu_0 \in \h, 
	\end{equation*}
	where $\rA$ and $G$ are defined as above and the $D(\rA^{-\frac12})$-valued nonlinear map $\mb$ is defined on $\h \times D(\rA^\frac12)$ or $D(\rA^\frac12 )\times \h$ by setting
	$$ \mb(\bu,\bv)= \lvert \bu \rvert \bv, $$ for any $(\bu,\bv)\in \h \times D(\rA^\frac12)$ or  $(\bu,\bv) D(\rA^\frac12 )\times \h$.
	It is clear that 
	\begin{align}
	\langle \rA \bv+\mb(\bu,\bv), \bv \rangle \ge \Vert \bv \Vert^2_\frac12,
	\end{align}
	for any $\bu, \bv \in \ve$. Here we should note that thanks to the solution of Kato's square root problem in \cite[Theorem 1]{Auscher}, see also \cite[Section 7]{Kenig}, we have $\Vert \bu \Vert_\frac12 \simeq \vert \nabla \bu \vert$ for any $\bu \in \h^1_0(\mo)$, i.e, $\ve=\h^1_0(\mo)$.
	
	Now we claim that for any numbers $\alpha\in [0,\frac12)$ and $\beta\in (0,\frac12)$ such that $\alpha+\beta \in (0,\frac12)$, there exists a constant $c_0>0$ such that 
	\begin{equation}\label{Claim}
	\Vert \mb(\bu,\bv)\Vert_{-\alpha}\le c_0 \begin{cases}
	\Vert \bu \Vert_{\frac12-(\alpha+\beta)} \Vert \bv \Vert_\beta\quad  \text{ for any } \bu \in \ve_{\frac12-(\alpha+\beta)}, \bv \in \ve_{\beta}\\
	\Vert \bu \Vert_{\beta} \Vert \bv \Vert_{\frac12-(\alpha+\beta)}\quad \text{ for any } \bv \in \ve_{\frac12-(\alpha+\beta)}, \bu \in \ve_{\beta},
	\end{cases}
	\end{equation}
	and 
	\begin{equation}\label{Claim-2}
	\Vert \mb(\bu,\bv)\Vert_{-\frac12}\le c_0 
	\Vert \bu \Vert_{\frac14} \Vert \bv \Vert_{\frac14}\quad \text{ for any } \bv \in \ve_{\frac14}, \bu \in \ve_{\frac14}.
	\end{equation}
	To prove these inequalities,  let $\beta> 0$ such that $\alpha +\beta < \frac12$. Since
	$$ \left(\frac12-\alpha\right) + \left(\frac12- 1+2(\alpha+\beta)\right)+ \left(\frac12-\beta\right)= 1,$$ we have 
	\begin{equation}\label{leb-ineq}
	\lvert \langle \vert \bu \vert \bv, \bw \rangle\rvert \le C_0 \Vert \bu \Vert_{\el^{r} } \Vert \bv\Vert_{\el^s} \Vert \bw \Vert_{\el^q}, 
	\end{equation} 
	where the constants $q,r,s$ are defined through 
	$$ 
	\frac1q=\frac12-\alpha, \;\; \frac1s=\alpha+\beta,\;\; \frac1r= \frac12-\beta.
	$$
	Recall that $\ve_\alpha\subset \h^{2\alpha}\subset \el^{q} $ with $\frac1q=\frac12- \alpha$ if $\alpha\in (0,\frac12)$ and $q\in [2,\infty)$ arbitrary if $\alpha=\frac12$. Then, we derive from \eqref{leb-ineq} that the second inequality in \eqref{Claim} holds. By interchanging the role of $r$ and $s$ we derive that the first inequality in \eqref{Claim} also holds.  One can establish \eqref{Claim-2} with the same argument. The estimates \eqref{Claim} and \eqref{Claim-2} easily {imply} \eqref{Eq-nonlin-B} and \eqref{Eq-nonlin-B-2}.
	
	Now we need to check that $\mb(\cdot, \cdot)$ satisfies \eqref{Eq-nonlin-B-H2}. For this purpose we observe that there exists a constant $C>0$ such that
	\begin{align*}
	\lvert \mb (\bu,\bv)\rvert\le C\lVert \bu \rVert_0 \lVert \bv \rVert_{\el^\infty},
	\end{align*}
	which with the continuous embedding $\ve_{\frac12+\eps}\subset \el^\infty$ for any $\eps>0$ implies \eqref{Eq-nonlin-B-H2}.

\end{document}